\documentclass[11pt]{amsart}
\usepackage{amssymb}
\usepackage{amsfonts}
\usepackage{amscd}
\usepackage{amsmath}
\usepackage{mathrsfs}
\usepackage{curves}
\usepackage{epsfig}
\usepackage[pass]{geometry}
\usepackage{graphicx}
\usepackage{verbatim}
\usepackage{latexsym}
\usepackage{eucal}

\usepackage{enumerate} 

\newtheorem{theorem}{Theorem}[section]
\newtheorem{lemma}[theorem]{Lemma}
\newtheorem{prop}[theorem]{Proposition}

\def \mcc {{\mathcal C}}
\def \mcd {{\mathcal D}}
\def \mce {{\mathcal E}}
\def \mcf {{\mathcal F}}

\def \mck {{\mathcal K}}

\def \mcm {{\mathcal M}}
\def \mcn {{\mathcal N}}

\def \mcp {{\mathcal P}}

\def \mcs {{\mathcal S}}

\def \mcu {{\mathcal U}}
\def \mcv {{\mathcal V}}

\def \mcx {{\mathcal X}}

\def \msh {{\mathscr H}}

\def \msu {{\mathscr U}}

\def \mbc {{\mathbb C}}

\def \mbr {{\mathbb R}}
\def \mbs {{\mathbb S}}

\def \id {\operatorname{Id}}
\def \comp {\operatorname{comp}}

\def \loc {\operatorname{loc}}

\def \diag{\textrm{Diag}}
\def \supp {\text{supp }}
\def \defeq {\stackrel{\operatorname{def}}{=}}
\def \beqq {\begin{equation}}
\def \eeqq {\end{equation}}

\def \bpf {\begin{proof}}
\def \epf {\end{proof}}
\def \beq {\begin{equation*}}
\def \eeq {\end{equation*}}
\def \eps {\epsilon}   
\def \la {\lambda}   
\def \La {\Lambda}    
   
\def \lap {\Delta}
\def \p {\partial}
\def \ha {\frac{1}{2}}

\begin{document}
\title[]{Inverse source problem for the Boltzmann equation in cosmology}
\author{Yiran Wang}
\address{Yiran Wang
\newline
\indent Department of Mathematics, Emory University
\newline
\indent 400 Dowman Drive, Atlanta, Georgia 30322, USA}
\email{yiran.wang@emory.edu}
\begin{abstract}
We study the inverse problem of recovering primordial perturbations from  anisotropies of Cosmic Microwave Background (CMB) using the kinetic model. Mathematically, the problem in concern is  the  inverse source problem for the linear Boltzmann equation with measurements on some Cauchy surface. We obtain two stable determination results  for generic absorption coefficients and scattering kernels. 
\end{abstract}
\date{\today.}
 
\maketitle


\section{Introduction}\label{sec-intro}
Consider the source problem for the linear Boltzmann equation (or non-stationary transport equation) on $\mcm = (0, T)\times \mbr^{3}, T>0$: 
\beqq\label{eq-boltz}
\begin{gathered}
\p_t u(t, x, \theta) + \theta \cdot \nabla_x u(t, x, \theta) + \sigma(t, x, \theta) u(t, x, \theta)   \\
= \int_{\mbs^{2}} k(t, x, \theta, \theta') u(t, x, \theta') d\theta' + f(t, x), 
\end{gathered}
\eeqq
where $t\in (0, T), x\in \mbr^3,  \theta\in \mbs^{2}$. Here, $\sigma$ is the absorption coefficient,  $k$ is the scattering kernel and $f$ is the source term.  We consider the zero initial condition 
\beqq\label{eq-boltz1}
u(0, x, \theta) = 0.
\eeqq
In this work, we study the inverse problem of determining the source term $f$ from the measurement of $u$ at  $t = T >0$
\beqq\label{eq-ut}
u(T, x, \theta) = u_T(x, \theta).
\eeqq 

The inverse problem for \eqref{eq-boltz} and its stationary version has a rich history, see Section 7.4 of \cite{Isa}. Both the determination of $\sigma, k$ and the source term $f$ have been investigated. In particular, there are lots of interest due to its application in optical imaging, see  for example review papers \cite{Bal, Ste}. Recently,  related inverse problems for the nonlinear Boltzmann equations have been studied, see \cite{BKL, LUY}.  Our perspective is somewhat different from previous works as our motivation comes from  inverse problems in cosmology.  We are interested in the determination of primordial gravitational perturbations from the anisotropies of the Cosmic Microwave Background (CMB), see \cite{KDS}. The physics background will be discussed in Section \ref{sec-cmb}.  The pure transport regime (namely without $\sigma, K$ in \eqref{eq-boltz})  serves as a good model for the standard universe after the decoupling time or the ``surface of last scattering". This was studied by Vasy and the author in \cite{VaWa} using the light ray transform, see also \cite{Wan2}. Before the decoupling time, photon interactions cannot be ignored and a kinetic model  based on the Boltzmann equation is appropriate. As is well-known in cosmology literatures e.g.\ \cite{Dod, Dur} (see also Section \ref{sec-cmb}), the linearization of the  Boltzmann equation on a Friedman-Lema\^ite-Robertson-Walker (FLRW) universe with respect to small metric perturbations naturally leads to a source problem for the Boltzmann equation as \eqref{eq-boltz} in which the source term is related to the metric perturbation. In fact, similar problems can also be considered in the relativistic kinetic theory such as the  Nordstr\"om-Vlasov system (see Section 4 of \cite{FJS}) on the linearization level, or even the Einstein-Boltzmann equations.

In this work, we obtain two results on the stable determination of the source term in \eqref{eq-boltz}. We  introduce some notations for stating the results. Consider the Minkowski spacetime $(\mbr^{3+1}, g)$ with signature $(-, +, +, +)$. Let $(\tau, \xi), \tau\in \mbr, \xi\in \mbr^3$ be the dual variables of $(t, x)$. 
We let $\Gamma^{tm}_\pm = \{(\tau, \xi)\in \mbr^{3+1}: \tau^2 > |\xi|^2, \pm \tau > 0\}$ be the set of future/past pointing time-like vectors. Let $\Gamma^{sp} =  \{(\tau, \xi)\in \mbr^{3+1}: \tau^2 < |\xi|^2\}$ be the set of space-like  vectors. Finally, let $\Gamma^{lt}_\pm = \{(\tau, \xi)\in \mbr^{3+1}: \tau^2 = |\xi|^2, \pm \xi_0 > 0\}$ be the set of future/past pointing light-like  vectors. We also let $\Gamma^{lt} = \Gamma_+^{lt}\cup \Gamma_-^{lt}$.  Let $\phi$ be the characteristic function of $\Gamma^{sp}$.  We define $\phi(D)$ to be  a Fourier multiplier $\phi(D) f = \mcf^{-1}(\phi \mcf f),  f\in L^2(\mbr^4)$ where $\mcf, \mcf^{-1}$ denote the Fourier and inverse Fourier transform in $t, x$ variables. We set $\mcv = (0, T)\times \Omega$ where $\Omega$ is a relatively compact set of $\mbr^3$ (so that $\mcv$ is a relatively compact set of $\overline\mcm$). {\em Throughout the paper, we assume that  $\sigma, k$ and $f$ are supported in $\mcv$}.  Our first result is  
\begin{theorem}\label{thm-main1}
Let $\sigma \in C^6$ be independent of the  $x$ and $\theta$ variable. There exists an open dense subset $\mcu$ of $C^6(\mcv\times \mbs^2\times\mbs^2)$ such that the following is true.  Consider the source problem \eqref{eq-boltz} and \eqref{eq-boltz1} with $k\in \mcu$ and $f\in H_{\comp}^2(\mcm)$. Then $f$ is uniquely determined by $u_T$ in \eqref{eq-ut}. Moreover,  we have the following stability estimate
\beqq\label{eq-thmsta}
\|\phi(D) f\|_{H^2(\mcm)} \leq C \|u_T\|_{H^{5/2}(\mbr^3 \times \mbs^2)}
\eeqq 
for some $C>0$ depending on $\sigma, k.$
\end{theorem}
The type of stability estimate \eqref{eq-thmsta} seems to be new and it is particularly important for our analysis. In fact, we will use the stability estimate to recover $\phi(D)f$ then use the analyticity of the Fourier transform of $f$ to prove the uniqueness.  

Next, for the CMB inverse problem, the metric perturbations that describe the evolution of the universe are not arbitrary. In fact, they are solutions of the linearized Einstein equations, see Section \ref{sec-cmb}. This leads us to study the inverse problem  of \eqref{eq-boltz} when the source $f$ is a solution of certain wave equations. For $s\geq 0$, we denote $\mcm_s = \{s\}\times \mbr^n$. Consider  
\beqq\label{eq-hyper}
P(z, \p) = \square + \sum_{j = 0}^n A_j(z)\p_j + B(z)
\eeqq
where $A_j, B$ are real or complex valued smooth functions in $z$.  Consider the Cauchy problem
\beqq\label{eq-cauchy}
\begin{gathered}
P(z, \p)f(z) = 0, \text{ on } \mcm   \\
u = f_1, \quad \p_t u = f_2 \text{ on } \mcm_0 
\end{gathered}
\eeqq  
Our second result is the stable determination of $f$ from $u_T.$ To avoid some technical issues, we will replace the source term $f$ in \eqref{eq-boltz} by $\chi_0 f$ where $\chi_0$ be a smooth cut-off function in $C_0^\infty((0, T))$ not identically vanishing. For some applications, it might be preferable to take $\chi_0$ as the characteristic function of $(0, T)$ in $\mbr$. However, by the well-posedness of the linear Boltzmann equation, the difference of $u_T$ can be made arbitrarily small in a proper sense. 
Below, we take $\mcv$ sufficiently large so that the solution of \eqref{eq-cauchy} with initial data supported in a fixe compact set $\mcx$ of $\mbr^3$ is contained in $\mcv.$
 \begin{theorem}\label{thm-main2}
Let $f$ be the solution of \eqref{eq-cauchy} on $\mcm$ with Cauchy data $f_1 \in H^{2}(\mcm_0), f_2\in H^{1}(\mcm_0)$ supported in a compact set $\mcx$ of $\mcm_0$ such that $f$ is supported in $\mcv.$ Suppose that   the coefficients $A_j(z)$ in \eqref{eq-hyper} are real valued smooth functions. 
Let $u$ be the solution of \eqref{eq-boltz}, \eqref{eq-boltz1} with source $\chi_0 f$. 
Then there exists an open dense set $\mcu$ of $C^\infty(\mcv\times \mbs^2)\times C^6(\mcv\times \mbs^2\times \mbs^2)$ such that for $(\sigma, k)\in \mcu$, $f_1, f_2$ is uniquely determined by $u_T$ and there exists $C >0$ such that 
\beqq\label{eq-thm2sta}
\|f\|_{H^{2}(\mcm)} \leq C \|(f_1, f_2)\|_{H^{2}(\mbr^3)\times H^{1}(\mbr^3)}\leq C \|u_T\|_{H^{5/2}(\mbr^3\times \mbs^2)}
\eeqq
\end{theorem} 

In Section \ref{sec-pf2}, we will  prove a stronger version of the theorem to include certain pseudo-differential operators which are motivated by the CMB inverse problem, see Section \ref{sec-cmb}. For $\sigma = k = 0$, Theorem \ref{thm-main2} was proved in \cite{VaWa} and further generalized in \cite{Wan2}.   

In the literature, there are some work on stability of the radiative transport equations based on the method of Carleman estimates, see \cite{Kli, MaYa}. For our problem, it seems  natural to follow the spirit in Stefanov and Uhlmann \cite{StUh} for the stationary transport equation to treat the map $f\rightarrow u_T$ as a perturbation of the light ray transform on the Minkowski spacetime. The difficulty is that, unlike the geodesic ray transform in the Riemmanian setting, the normal operator of the light ray transform is not an elliptic pseudo-differential operator. In fact, the Schwartz kernel belongs to the class of paired Lagrangian distributions, see \cite{Wan1}. The key of our approach is to restore the ellipticity by using either $\phi(D)$ or the parametrix of the Cauchy problem.

We have a few remarks. First, our results should hold for general dimensions, however we study $\mbr^{3+1}$ for its physical relevance.   
Second, as we assume that $f$ is compactly supported in $[0, T]\times \Omega$, one can consider the problem with measurements  on the lateral boundary $[0, T]\times \p \Omega$ using the method we develop here. The problem then is the time-dependent version of the inverse source problem studied in  \cite{StUh}.  Finally, the stability estimates  suggest  that our results can be generalized via  perturbation arguments to other scenarios such as small metric perturbations of the Minkowski spacetime as in \cite{VaWa}, small perturbations of $\sigma$ for Theorem \ref{thm-main1} and nonlinear perturbations in the Boltzmann equation.

The paper is organized as follows. We begin with a discussion of the CMB kinetic theory in Section \ref{sec-cmb} and derive the inverse source problem. In Section \ref{sec-direct}, we prove the solvability of the forward problem via the Fredholm theory. We prove Theorem \ref{thm-main1} in Section \ref{sec-pfthm1} after collecting some results of the Minkowski light ray transform in Section \ref{sec-lray}. The last three sections are devoted to the proof of Theorem \ref{thm-main2}. We first study the light ray transform with weights in Section \ref{sec-lraywei}. Then we analyze the action of the transform on solutions of the Cauchy problem in Section \ref{sec-comp}. Finally, we prove Theorem \ref{thm-main2} in Section \ref{sec-pf2}.

{\em Acknowledgement:} This work is supported by National Science Foundation under grant DMS-2205266.

\section{The CMB kinetic theory}\label{sec-cmb}
In this section, we discuss the inverse problem of determining primordial perturbations from the anisotropies of CMB. Our goal is to show how the source problem for the Boltzmann equation naturally appears and how the source term is connected to the metric perturbations.  We will consider a simple setup,  emphasizing more on the mathematical structure of the problem. The perturbation theory for CMB anisotropies has been well-developed in cosmology literatures, which can be found in \cite{Dod, Dur} for instance. 

\begin{figure}[htbp]
\centering
\vspace{-5cm}
\includegraphics[scale = 0.55]{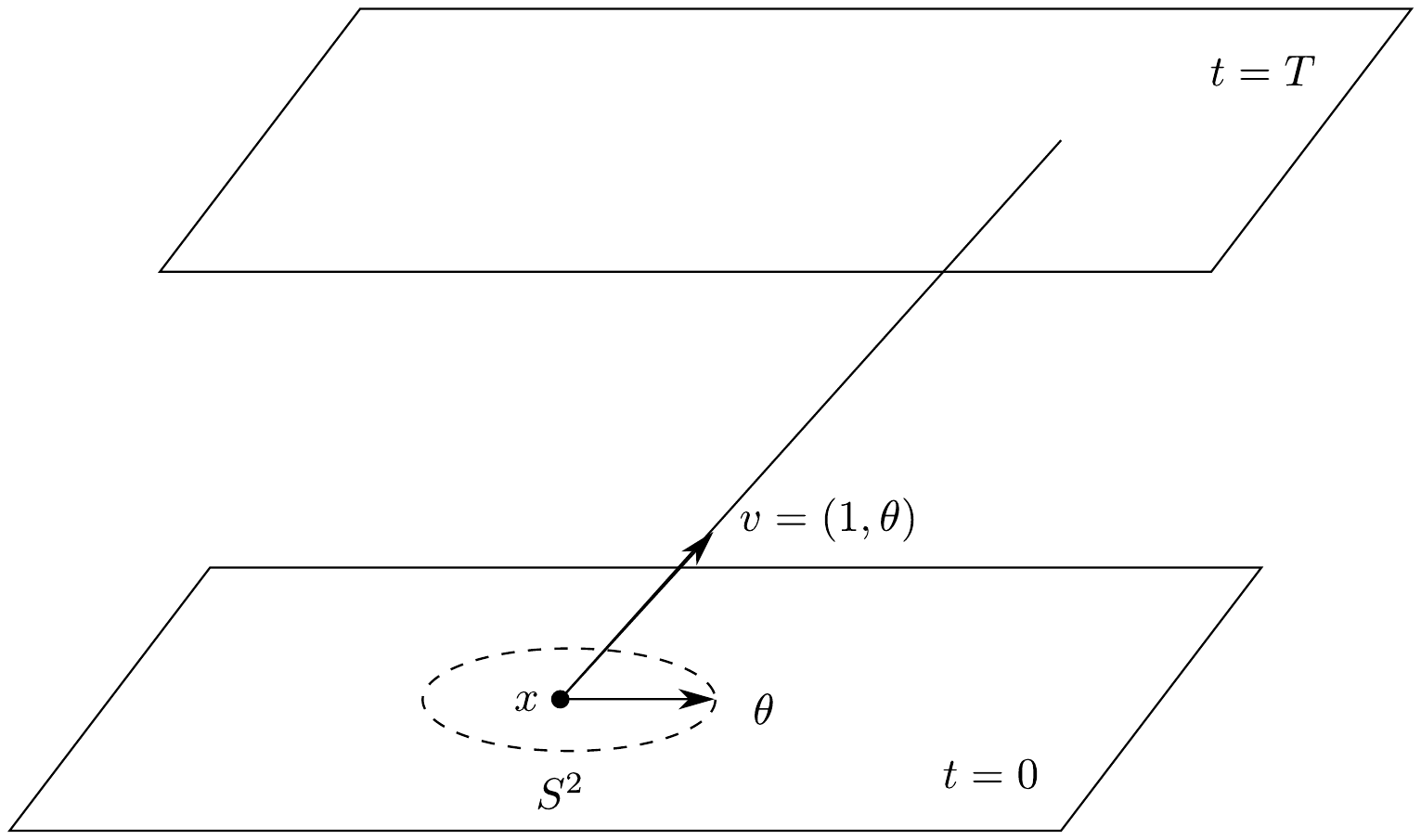}
\vspace{-6cm}
\caption{Parametrization of the light rays}
\label{fig-lray}
\end{figure}

Consider the FLRW spacetime $(\mcm, g)$ as the Universe model, where $\mcm = [0, \infty)\times \mbr^3$ and $g = -dt^2 + a^2(t) dx^2$ with $a>0.$ Note that $(\mcm, g)$ is conformal to the Minkowski spacetime and the conformal transformation does not change much of the analysis. So we work with the Minkowski spacetime $(\mcm, g)$ below (by simply taking $a = 1$). Let $\Phi, \Psi \in C^\infty(\mcm)$. For $\eps>0$ small, we consider a smooth family of Lorentzian metrics on $\mcm$
\beqq\label{eq-metric}
g_\eps = (-1 + \eps \Phi)dt^2 + (1 + \eps \Psi)dx^2 + \eps^2 h_\eps
\eeqq
where $h_\eps$ is a family of symmetric two tensors on $\mcm$ smooth for $\eps\in [0, \eps_0), \eps_0 >0.$ Later, we also use $z = (z_0, z_1, z_2, z_3) = (t, x_1, x_2, x_3)$ for local coordinates. Note that $g_0$ is the Minkowski metric, and we regard $(\mcm, g_\eps)$ as the perturbation of the Minkowski spacetime.  

Consider light-like geodesics $\gamma_\eps(s), s \geq 0$ on $(\mcm, g_\eps)$ originating  from $\mcm_0 = \{0\}\times \mbr^3$ which we think of as photon trajectories. They satisfy the geodesic equation
\beqq\label{eq-geod}
\ddot \gamma_\eps^k(s) + \Gamma_{\eps, ij}^k(s) \dot \gamma_\eps^i(s) \dot \gamma_\eps^j(s) = 0
\eeqq
with initial conditions
\beq
\gamma_{\eps}(0) =  \tilde z, \quad \dot \gamma_{\eps}(0) = \tilde \zeta.
\eeq
Here $\Gamma_{\eps, ij}^k(s)$ denotes the Christoffel symbols for $g_\eps$ along $\gamma_\eps(s)$.  Let $p_\eps^i(s) = \dot \gamma_\eps^i(s), i = 1, 2, 3$ be the momentum and $p_\eps^0(s) = \dot \gamma_\eps^0(s)$ be the energy of the photon. 
In particular, $p_\eps = (p^i_\eps)_{i = 0}^3$ is a  vector field along $\gamma_\eps.$ As we consider light-like geodesics for (massless) photons, $p_\eps$ are (future pointing) null vectors, namely $g_\eps(p_\eps, p_\eps) = 0$ 
along $\gamma_\eps$.   
It is convenient to denote $v = p_\eps^0 >0$ the energy and  $\theta^i = p_\eps^i/p_\eps^0, i = 1, 2, 3$.  In particular, we have $p_\eps = v(1, \theta)$.

Now let $f_\eps$ be the photon distribution function which  is  a function of $z, p$ variables where $z\in \mbr^{3+1}$ and $p$ is on the mass shell 
\beq
\Sigma_z = \{p \in T_z\mbr^{3+1}: g_\eps(p, p) = 0\}.
\eeq 
 We assume that $f_\eps$ satisfies the linear Boltzmann equation, see \cite[Section 4.5]{Dur}. This means that along $\gamma_\eps$ 
  \beqq\label{eq-bol}
 \frac{d }{ds}f_\eps(\gamma_\eps(s), p_\eps(s)) = C[f_\eps]
 \eeqq
 where $C[f]$ denotes the interaction term  
 \beqq\label{eq-intera}
 C[f] =- \sigma(z) f(z, p) + \int k(z, \theta, \theta') f(z, v(1, \theta'))d\theta'
 \eeqq
 where $\sigma$ denotes absorption coefficients, $k$ is the scattering kernel and the integration is over $\{\theta: v(1, \theta) \in \Sigma_z \text{ for } v > 0\}$. 
 The terms  in \eqref{eq-intera} accounts for photon interactions in Thomson scattering for example. We get from \eqref{eq-bol} and \eqref{eq-intera} the equation 
 \beqq\label{eq-boltz0}
 \begin{gathered}
  \sum_{i = 0}^3 \frac{\p f_\eps}{\p z^i}(z, p) \frac{\p \gamma^i_\eps}{\p s} +   \frac{\p f_\eps}{\p p}(z, p) \frac{\p p_\eps}{\p s} 
 = - \sigma(z) f_\eps(z, p) + \int k(z, \theta, \theta') f_\eps(z, v(1, \theta'))d\theta'
 \end{gathered}
 \eeqq
 Now we consider  $f_\eps$ as a perturbation of some background distribution with an expansion 
 \beqq\label{eq-asymp}
 f_\eps(z , p) = f_0(v) + \eps f_1(z, v, \theta) + O(\eps^2) 
 \eeqq
 Here, $f_0$ is the background photon distribution. When modeling the cosmic microwave background, one often assumes that  $f_0$ satisfies the Planck distribution  
 \beq
 f_0(v) = (e^{v/T_0} + 1)^{-1} 
 \eeq
 see page 149 of \cite{Dur}. Here, $T_0 >0$ be the background temperature of the universe.  $f_1$ in \eqref{eq-asymp} is the first order perturbation term and $\theta$ is taken over $\mbs^2$. In particular, $(1, \theta)$ is a future pointing light-like vector for the background Minkowski metric $g_0.$

 We find the $\eps$ derivative of the equation \eqref{eq-boltz0} at $\eps = 0$. 
  \beqq\label{eq-boltznew1}
  \begin{gathered}
 v \frac{\p f_1}{\p t}(z, v, \theta)  
  + v \sum_{j = 1}^3 \theta_j \frac{\p f_1}{\p z^j}(z, v, \theta) + \frac{\p f_1}{\p p_0^j} (z, v, \theta)\frac{\p p^j_0}{\p s}   \\
  +  \frac{\p f_0}{\p v}(v)   \p_\eps\frac{\p p^0_\eps}{\p s}|_{\eps = 0}   
   = - \sigma(z)  f_1(z, v, \theta)   + \int_{\mbs^2} k(z, \theta, \theta') f_1(z, v, \theta')d\theta'
  \end{gathered}
 \eeqq
 Here, $p_0(s) = \dot \gamma_0(s)$ is the vector filed along the geodesic $\gamma_0$ for the background metric. We observe that $\p_s p^j_0(s) = 0$ for $j = 0, 1, 2, 3$ which can be seen from the geodesic equation on $(\mbr^{3+1}, g_0)$ and the fact that $g_0$ is a constant metric.  
It remains to find $\p_\eps \frac{\p p^0_\eps}{\p s}|_{\eps = 0}$ in \eqref{eq-boltz1}. We use the geodesic equation \eqref{eq-geod} on $(\mbr^{3+1}, g_\eps)$
 \beq
 \dot p_\eps^0(s) + \Gamma_{\eps, ij}^0(\gamma_\eps(s)) p_\eps^i(s) p_\eps^j(s)  = 0
 \eeq
 thus 
 \beq
 \begin{gathered}
 \p_\eps (\frac{\p p_\eps^0}{\p s})|_{\eps = 0} = \p_\eps\Gamma_{\eps, ij}^0|_{\eps = 0} p^i_0p^j_0
 \end{gathered}
 \eeq
 In the calculation, we used the fact that the Christoffel symbols for the Minkowski spacetime all vanish. To find the linearization of the Christoffel symbol, recall that 
 \beq
 \Gamma_{jk}^i  = \ha g^{i\mu} (\frac{\p g_{\mu k}}{\p z^j} - \frac{\p g_{j k}}{\p z^\mu}  + \frac{\p g_{j \mu}}{\p z^k} )
 \eeq
 Therefore,
 \beq
 \begin{gathered}
 \p_\eps \Gamma_{\eps, ij}^0|_{\eps = 0}   =0 \text{ if $i\neq j$, and $i, j\neq 0$}\\
 \p_\eps \Gamma_{\eps, jj}^0|_{\eps = 0} = +\ha \frac{\p \Psi}{\p z^0}  \text{ if $j \neq 0$} \text{ and } 
   \p_\eps \Gamma_{\eps, 0j}^0|_{\eps = 0} =  -\ha \frac{\p \Phi}{\p z^j}  
 \end{gathered}
 \eeq
We deduce that 
  \beq
  \begin{gathered}
\p_\eps\frac{\p p^0_\eps}{\p s}|_{\eps = 0}    =   \ha \frac{\p \Psi}{\p z^0}  p^j_0p^j_0  - \ha \frac{\p \Phi}{\p z^j} p^0_0p^j_0
\end{gathered}
  \eeq
Using this in \eqref{eq-boltznew1} we get
  \beqq\label{eq-boltz2} 
  \begin{gathered}
 v\frac{\p f_1}{\p t}(z, v, \theta)  
  +  v\theta^j \frac{\p f_1}{\p x^j}(z, v, \theta)     +  \sigma(z)  f_1(z, v, \theta)   \\
  - \int_{\mbs^2} k(z, \theta, \theta') f_1(z, v, \theta')d\theta' = -  \frac{\p f_0}{\p v}(v) v  ( \ha \frac{\p \Psi}{\p t}  v - \ha \frac{\p \Phi}{\p z^j} v\theta^j)
  \end{gathered}
 \eeqq 
At this point, we will take $v >0$ to be a fixed constant and ignore it in $f_1$. We get from \eqref{eq-boltz2} that 
     \beqq\label{eq-boltz3}
  \begin{gathered}
 \frac{ \p f_1}{\p t}(z, \theta)  
  +  \theta^j \frac{\p f_1}{\p z^j}(z, \theta)     +  \sigma(z)  f_1(z, \theta)   - \int_{\mbs^2} k(z, \theta, \theta') f_1(z, \theta')d\theta'\\
= C ( \ha \frac{\p \Psi}{\p t}  - \ha \frac{\p \Phi}{\p z^j} \theta^j)
  \end{gathered}
 \eeqq
 where $C$ is a non-zero constant and $\sigma, k$ are changed by a scalar factor. This is essentially the Boltzmann equation we considered in the introduction, and the inverse problem is to determine $\Phi, \Psi$ from the observation of $f_1$ at $t = T.$  
 We remark that in cosmology literatures, one integrates \eqref{eq-boltz2} in $v$ and obtain an equation for a quantity independent of $v$. The quantity is related to the temperature perturbation or the redshift of the CMB, see Section 4.5.1 of \cite{Dur} for instance. However, the mathematical structure of the so-derived equation is identical to \eqref{eq-boltz3}.  Actually, it is more convenient to change $f_1$ in \eqref{eq-boltz3} to $\tilde f_1 = f_1 + \ha C\Phi$. Then we find from \eqref{eq-boltz3} that 
      \beqq\label{eq-boltz4}
  \begin{gathered}
 \frac{ \p \tilde f_1}{\p t}(z, \theta)  
  +  \theta^j \frac{\p \tilde f_1}{\p z^j}(z, \theta)     +  \sigma(z)  \tilde f_1(z, \theta)   - \int_{\mbs^2} k(z, \theta, \theta') \tilde f_1(z, \theta')d\theta'\\
= C ( \ha \frac{\p \Psi}{\p t}  + \ha \frac{\p \Phi}{\p t} + B(z)\Phi )
  \end{gathered}
 \eeqq
 where $B(z)$ is a smooth function of $z$ depending on $\sigma, k.$ When $\sigma = k = 0$, $B(z) = 0.$ The difference of $f_1$ and $\tilde f_1$ is independent of the direction and is in fact not measurable from CMB anisotropies, see \cite[Chapter 4]{Dur}.

Finally, let's consider the metric perturbations  $\Phi, \Psi$ in \eqref{eq-metric}. When modeling the evolution of the universe, one assumes that $g_\eps$ satisfies the Einstein equations with matters. In case of scalar fields matter, the linearized term $\Phi, \Psi$ are known to be equal and satisfy the Bardeen's equation, which is of the form 
\beqq\label{eq-bar}
\p_t^2 \Psi  +  A_0(t) \p_t \Psi +  \lap \Psi  + B_0(t)\Psi = 0. 
\eeqq 
See equation (6.48) of \cite{MFB}. Here, $A_0(t), B_0(t)$ are smooth functions and $\lap$ is the Laplacian on $\mbr^3$. The inverse problem now is to determine $\Psi$ in $\mcm$ satisfying \eqref{eq-bar} with measurement of $\tilde f_1$ of \eqref{eq-boltz4} at $t = T$. Note that this problem fits Theorem \ref{thm-main2} except that the source term in \eqref{eq-boltz4} involves an extra differential operator. This will be treated in the stronger version Theorem \ref{thm-main3} of Theorem \ref{thm-main2} in Section \ref{sec-pf2}.

\section{Solvability of the direct problem}\label{sec-direct}
For the solvability of the source problem, we will follow  \cite{StUh} to give a proof based on the analytic Fredholm theory. Compared with the results for the stationary transport equation in \cite{StUh}, we need higher regularity requirements for $\sigma, k$ and $f$. We assume $n\geq 2$ in this section.

\begin{theorem}\label{thm-exist}
For $\sigma\in C^5(\mcv)$, there exists an open and dense subset $\msu$ of $C^5(\mcv \times \mbs^{n-1}\times \mbs^{n-1})$ such that for any $k \in  \msu$ and $f\in H_{\comp}^2(\mcm)$, the equation \eqref{eq-boltz} with initial condition \eqref{eq-boltz1} has a unique solution $u\in H^2(\mcm \times \mbs^{n-1}).$
\end{theorem}
 
 We remark that the theorem can also be stated for $(\sigma, k)$ in an open dense subset $\msu$ of $C^5(\mcv)\times C^5(\mcv \times \mbs^{n-1}\times \mbs^{n-1})$. 
 
For the proof, we let 
\beq
T_0 = \p_t  + \theta \cdot \nabla_x, \quad T_1 = T_0 + \sigma, \quad T = T_1 - K
\eeq
where $\sigma$ is regarded as the multiplication operator and $K$ is the integral operator in \eqref{eq-boltz}. For $k = 0$, the equation $T_1 u = f$ with $u = 0$ at $t = 0$ can be solved explicitly. For $\theta\in \mbs^{n-1}, t > 0, x\in \mbr^n$, consider $u(t, x, \theta) = u(t, x + t\theta)$ which satisifes 
\beqq\label{eq-u}
\frac{d}{dt} u(t, x + t\theta) + \sigma(t, x + t\theta) u(t, x + t \theta) = f(t, x + t\theta)
\eeqq
An integrating factor is $E(t, x, \theta) = e^{\int_0^t \sigma(s, x + s\theta)ds}$. We solve \eqref{eq-u} that 
\beq
\begin{gathered}
u(t, x + t\theta) = e^{-\int_0^t \sigma(s, x + s\theta)ds} \int_{0}^t  e^{\int_0^s \sigma(\tilde s, x + \tilde s\theta)d\tilde s} f(s, x+s\theta)ds\\
 = \int_{0}^t  e^{- \int_s^t \sigma(\tilde s, x + \tilde s\theta)d\tilde s} f(s, x+s\theta)ds
\end{gathered}
\eeq
Thus we can write $T_1^{-1}$ as 
\beqq\label{eq-T1}
\begin{gathered}
 T_1^{-1} f(t, x, \theta) = \int_{0}^t  \kappa(t, x, s, \theta)f(s, x+s\theta)ds,  \\
\text{with } \kappa(t, x, s, \theta) = e^{- \int_s^t \sigma(\tilde s, x + \tilde s\theta)d\tilde s} 
 \end{gathered}
\eeqq
Next, for $T u = (T_1 - K) u = f$, we apply $T_1^{-1}$ and get $(\id - T_1^{-1} K) u = T_1^{-1} f$. The main part of the proof is to show that $\id - T_1^{-1} K$ is invertible for suitable $k$ so that 
\beqq\label{eq-sol}
u = (\id - T_1^{-1}K)^{-1} T_1^{-1} f
\eeqq
Notice that this can be written as 
\beqq\label{eq-sol1}
u = T_1^{-1}(\id - K T_1^{-1})^{-1} f.
\eeqq
We will show that $\id - K T_1^{-1}$ is invertible. As in \cite{StUh}, we introduce 
\beq 
A = (\id - ( K T_1^{-1})^2)^{-1}
\eeq
and write
\beq 
(\id - K T_1^{-1})^{-1} = (\id + K T_1^{-1})A. 
\eeq
We will show that $(K T_1^{-1})^2$ is compact and apply the analytic Fredholm theory to conclude that $A$ is invertible.

For the proof, we will need the following proposition and its variant about singular operators.
\begin{prop}[Proposition 3.4 of \cite{StUh}]\label{prop-sinest}
Let $A$ be the operator
\beq
Af(x) = \int \frac{\alpha(x, y, |x - y|, \frac{x-y}{|x - y|})}{|x - y|^{n-1}} f(y) dy
\eeq
with $\alpha(x, y, r, \theta)$ compactly supported in $x, y\in \mbr^n$.  
\begin{enumerate}[(i)]
\item
 If $\alpha\in C^2$, then $A: L^2(\mbr^n)\rightarrow H^1(\mbr^n)$ is continuous with a norm not exceeding $C\|\alpha\|_{C^2}$.
\item Let $\alpha(x, y, r, \theta) = \alpha'(x, y, r, \theta)\phi(\theta)$, then 
\beq
\|A\|_{L^2\rightarrow H^1} \leq C\|\alpha'\|_{C^2} \|\phi\|_{H^1(\mbs^{n-1})}
\eeq
\end{enumerate}
\end{prop}
We remark that the constant $C$ is independent of $\|\alpha\|_{C^2}$ but   depends on the support of $\alpha.$ It can be made uniform if $\alpha(x, y, \cdot, \cdot)$ is supported in a fixed compact set in $x, y$ variables.  
The proposition can be slightly improved for $H^m, m\geq 0$ functions. 
\begin{prop}\label{prop-sinest1}
Consider operator $A$ in Proposition \ref{prop-sinest}.  
\begin{enumerate}[(i)]
\item
 If $\alpha\in C^{m+2}, m  = 0, 1, 2, \cdots$, then $A: H^m(\mbr^n)\rightarrow H^{m+1}(\mbr^n)$ is continuous with a norm not exceeding $C\|\alpha\|_{\tilde C^{m+2}}$ where
 \beq
 \|\alpha\|_{\tilde C^{m+2}} = \sup \sum_{|\gamma| = 2, |\beta| = m}|\p_{x, y}^\beta \p_{x, y, r, \theta}^\gamma \alpha|
 \eeq
\item Let $\alpha(x, y, r, \theta) = \alpha'(x, y, r, \theta)\phi(\theta)$, then 
\beq
\|A\|_{H^m\rightarrow H^{m+1}} \leq C\|\alpha'\|_{\tilde C^{m+2}} \|\phi\|_{H^1(\mbs^{n-1})}
\eeq
\end{enumerate}
\end{prop}
\bpf
Assume that $f\in H_{\comp}^m(\mcm)$. We prove for $m=1$ and the other cases are similar.  Using polar coordinate, we write 
\beq
\begin{gathered}
Af(x) = \int \alpha(x, x + r\theta, r, \theta) f(x + r\theta) dr d\theta
\end{gathered}
\eeq
For $i = 1, 2, \cdots, n$, we get
\beq
\begin{gathered}
\p_{x_i}Af(x) = \int \p_{x^i}\alpha(x, x + r\theta, r, \theta) f(x + r\theta) + \alpha(x, x + r\theta, r, \theta) \p_{x^i} f(x + r\theta) dr d\theta\\
 =  \int   \frac{(\p_{x_i}\alpha + \p_{y_i}\alpha)(x, y, |x - y|, \frac{x-y}{|x - y|})}{|x - y|^{n-1}}  f(y) dy 
 + \int  \frac{\alpha(x, y, |x - y|, \frac{x-y}{|x - y|})}{|x - y|^{n-1}}  \p_{y^i} f(y) dy
\end{gathered}
\eeq
Now we can apply Proposition \ref{prop-sinest} to finish the proof. 
\epf

\begin{lemma}\label{lm-cpt}
The operator $K T_1^{-1}K$ is compact  on  $H^2(\mcv\times \mbs^{n-1})$.
\end{lemma}
\bpf
 We find that
\beq
\begin{gathered}
K T_1^{-1} f(t, x, \theta) = \int_{\mbs^{n-1}}k(t, x, \theta, \theta') \int_{0}^t  \kappa(t, x, s, \theta') f(s, x+s\theta', \theta')ds d\theta'
\end{gathered}
\eeq
where $\kappa$ is defined in \eqref{eq-T1}. 
Set $y = x + s\theta'$. We get $s = |y - x|, \theta' = (y - x)/|y - x|$ and 
\beq
\begin{gathered}
K T_1^{-1} f(t, x, \theta) = \int_{\mcv} \frac{k(t, x, \theta, \frac{y-x}{|y-x|}) \kappa(t, x, |y-x|, \frac{y-x}{|y-x|})}{|y - x|^{n-1}} f(|y-x|, y, \frac{y-x}{|y-x|})dy
\end{gathered}
\eeq
Next, we compute
\beqq\label{eq-KTK}
\begin{gathered}
K T_1^{-1} K f(t, x, \theta) = \int \frac{\alpha(t, x, y, \theta, \theta')}{|y - x|^{n-1}}  f(|y-x|, y, \theta')dyd\theta'
\end{gathered}
\eeqq
where 
\beq
\alpha(t, x, y, \theta, \theta') = k(t, x, \theta, \frac{y-x}{|y-x|}) \kappa(t, x, |y-x|, \frac{y-x}{|y-x|})k(|y-x|, y, \frac{y-x}{|y-x|}, \theta') 
\eeq
Note that $\alpha$ is $C^5$. Since $f$ is compactly supported in the $t$ variable, we can write $f$ in Fourier  series in $t$ as 
\beq
f(t, x, \theta) = \sum_{n = -\infty}^\infty f_n(x, \theta) e^{i2 \pi n t/T},
\eeq
where $f_n(x, \theta) = \frac{1}{T}\int_0^T f(t, x, \theta) e^{-i2 \pi n t/T} dt$. Also, for $f\in  H^2(\mcv\times \mbs^{n-1})$, Plancherel's theorem tells 
\beq
\|f\|^2_{H^2(\mcm\times \mbs^{n-1})} = \sum_{n=-\infty}^\infty \|f_n \|^2_{H^2(\mbr^n\times \mbs^{n-1})}
\eeq
and 
\beq
\|\p_t^2 f\|^2_{H^2(\mcm\times \mbs^{n-1})} = \sum_{n=-\infty}^\infty T^{-4} n^4 \|f_n \|^2_{H^2(\mbr^n\times \mbs^{n-1})}
\eeq
See for example \cite{Tor}. Here, $f_n$ are functions on $\Omega\times \mbs^{n-1}$ and we extended them trivially to $\mbr^n\times \mbs^{n-1}$ for convenience. 
Let $g_n(t, x, \theta) = f_n(x, \theta) e^{i2 \pi n t/T} $. We have 
\beq
\begin{gathered}
K T_1^{-1} K g_n (t, x, \theta) = \int \frac{\alpha(t, x, y, \theta, \theta')}{|y - x|^{n-1}} e^{i2 \pi n  \frac{ |y - x|}{T}} f_n(y, \theta')dyd\theta'
\end{gathered}
\eeq
For fixed $t$, it follows from Proposition \ref{prop-sinest1} (i) that the operator is bounded from $H^2(\mbr^n \times \mbs^{n-1})$ to $H^3(\mbr^n\times \mbs^{n-1})$ with norm not exceeding $C n^2 \|\alpha\|_{C^4}$ with $C$ depending on $\mcv$, namely
\beq
\|K T_1^{-1} K g_n\|_{H^3(\mbr^n \times \mbs^{n-1})} \leq C \|\alpha\|_{C^4} n^2 \|f_n\|_{H^2(\mbr^n\times \mbs^{n-1})}
\eeq
Note that when applying Proposition \ref{prop-sinest}, we need $k$ to be compactly supported in $t, x$ variable. Summing up in $n$, we get 
\beq
\|K T_1^{-1} K f(t, \cdot, \cdot)\|^2_{H^3(\mbr^n \times \mbs^{n-1})} \leq C \sum_{n=-\infty}^\infty n^4 \|f_n\|^2_{H^2(\mbr^n\times \mbs^{n-1})} \leq C \|f\|^2_{ H^2(\mcm\times \mbs^{n-1})}
\eeq
under our regularity assumption on $f$. This shows in particular that $K T_1^{-1} K f\in L^2([0, T], H^3(\mbr^3\times \mbs^{n-1})).$ 
By considering $\p_{t}^\beta (K T_1^{-1} K f)$ for $|\beta|\leq 3$, we see that $KT_1^{-1}K$ is bounded from $H^2(\mcv\times \mbs^{n-1})$ to $ H^3(\mcv\times \mbs^{n-1})$ by using Proposition \ref{prop-sinest1} and $\sigma, k \in C^5. $ Thus $KT_1^{-1}K$ is compact  on $H^2(\mcv\times \mbs^{n-1})$.
\epf

\bpf[Proof of Theorem \ref{thm-exist}]  We aim to find $k$ such that $T^{-1}$ exists. Let $\la\in \mbc$. We replace the scattering kernel $k$ in \eqref{eq-boltz} by $\la k$ and denote the corresponding operator by $\la K.$ Formally, we consider 
\beqq\label{eq-inv1}
A(\la) = (\id - (\la K T_1^{-1})^2)^{-1}
\eeqq
so that 
\beqq\label{eq-inv2}
(\id - \la K T_1^{-1})^{-1} = (\id + \la K T_1^{-1})A(\la). 
\eeqq
We need to justify the invertibility of the operator in $A(\la)$. 
Since $(\la KT_1^{-1})^2$ is compact from Lemma \ref{lm-cpt}, by the analytic Fredholm theorem \cite[Theorem VI.14]{ReSi1}, we know that there exist a discrete set $\mcs$ of $\mbc$ such that for $\la\notin \mcs$, $A(\la)$ exits and for such $\la$, \eqref{eq-inv2} is justified. We now use \eqref{eq-sol1} and that $T_1^{-1}$ is bounded on $H^2(\mcm)$ which might not be optimal. But this shows that the operator $T$ is invertible on $H^2(\mcm)$ for scattering kernel $\la k$ where $\la\in \mbc\backslash \mcs$ which implies that the set of such $k$ is dense in $C^5(\mcm\times \mbs^{n-1}\times \mbs^{n-1})$.
\epf
 
Let $u$ be the solution of \eqref{eq-boltz} with initial condition \eqref{eq-boltz1}. We set 
\beqq\label{eq-X}
Xf=  u|_{t = T}.
\eeqq
We can use Theorem \ref{thm-exist} to obtain a representation for $X$. Let $\rho_T$ be the restriction operator to $t = T.$ It follows from \eqref{eq-sol} that
\beqq\label{eq-X1}
X = \rho_T T_1^{-1} (\id - K T_1^{-1})^{-1}  
\eeqq
We use that $\rho_T T_1^{-1}$ is a light ray transform with weight and is  bounded from $H^2(\mcm)$ to $H^{5/2}(\mbr^n \times \mbs^{n-1})$, see Section \ref{sec-lraywei}. We  conclude that  $X: H^2(\mcm) \rightarrow H^{5/2}(\mbr^n \times \mbs^{n-1})$ is bounded. 

\section{The Minkowski light ray transform}\label{sec-lray}
To prove Theorem \ref{thm-main1}, we will treat $X$ in \eqref{eq-X1} as a perturbation of the light ray transform on Minkowski spacetime by compact operators.  In fact, when $\sigma = k = 0$, we see that
\beq
X f = \rho_T T_0^{-1} f = \int_{0}^T  f(s, x+s\theta)ds
\eeq
which is basically the light ray transform on  the Minkowski spacetime $(\mbr^{3+1}, g)$ where $g = -dt^2 + dx_1^2 + dx_2^2 + dx_3^2$.  We parametrize the light ray transform using null vectors at $\mcm_0$ as follows, see Figure \ref{fig-lray}.  
For $\theta\in \mbs^{2}, x\in \mbr^3$, the light-like geodesics from $(0, x)$ in the direction $(1, \theta)$ is given by $l_{x, \theta}(s) = (s, x + s\theta),  s\in \mbr.$  The set of light rays are parametrized by the set $\mcc \defeq \mbr^3\times \mbs^{2}.$   We parametrize the light ray transform  as 
\beqq\label{eq-lray}
L  f(x, \theta) = \int_\mbr  f(s, x + s\theta) ds
\eeqq
 When $\sigma\neq 0, k = 0$, $X$ is a light ray transform with weight which we study later in Section \ref{sec-lraywei}. If $\sigma(z) = \sigma(t)$ only depends on the $t$ variable, we have
\beqq\label{eq-X0}
\begin{gathered}
X f 
= \int_{0}^T \kappa(s) f(s, x+s\theta)ds = L(\kappa f)  
\text{ where } \kappa(s) =  e^{- \int_s^T \sigma(\tilde s)d\tilde s} 
\end{gathered}
\eeqq
In this case, it suffices to look at the light ray transform $L$. This is why we impose the assumption on $\sigma$ in Theorem \ref{thm-main1}.

It is known that $L$ is injective on $C_0^\infty$ functions. However, when acting on say Schwartz functions, $L$ has a non-trivial kernel consisting of functions whose Fourier transform is supported in $\Gamma^{tm}$, see for instance \cite{Ilm}.  Recall the operator $\phi(D)$ defined in the introduction.   Let $\phi \in \mcd'(\mbr^{3+1})$ such that $\phi(\zeta) = 1$ if $\zeta \in \Gamma^{sp}\cup \Gamma^{lt}\cup \{0\}$ and $\phi(\zeta) = 0$ if $\zeta\in \Gamma^{tm}$.  
It is easy to see that $\phi(D): H^s(\mbr^{3+1}) \rightarrow H^s(\mbr^{3+1}), s\in \mbr$ is bounded. Also, $\phi^2(D) = \phi(D)$ so $\phi(D)$ is a projection on $H^s(\mbr^{3+1})$. We denote the range of $\phi(D)$ on $H^s(\mbr^{3+1})$ by $\msh^s$ which is a closed subspace of $H^s(\mbr^{3+1})$, hence a Hilbert space.  The following simple result plays an important role.  
\begin{prop}\label{prop-cut}  
 For $f\in L^2_{\comp}(\mbr^{3+1}),$ we have 
\beqq\label{eq-fou}
L (\kappa f) = L (\phi(D) (\kappa f)).
\eeqq
\end{prop}
 \bpf
We use the Fourier slice theorem
\beqq\label{eq-fou1}
\begin{gathered}
\mcf_{y}(L (\kappa f))(\xi, \theta)   = \int_{\mbr^n} \int_{\mbr} e^{-iy\xi} \kappa(s)f(s, y + s\theta)dsdy  \\
= \mcf_{(t, x)} (\kappa f ) (-\theta \cdot \xi, \xi, \theta)
\end{gathered}
\eeqq
Here, $\kappa f\in L^2_{\comp}$. We apply the Fourier slice theorem to $L\phi(D) \kappa f$ and get 
\beqq\label{eq-fou2}
\begin{gathered}
\mcf_{y}(L  (\phi(D)\kappa f))(\xi, \theta)   = \phi(-\theta \cdot \xi, \xi )\mcf_{(t, x)} (\kappa f ) (-\theta \cdot \xi, \xi, \theta)
\end{gathered}
\eeqq
For any $\theta\in \mbs^{2}$, we see that $(-\theta \cdot \xi, \xi)$ is non time-like for all $\xi\in \mbr^3\backslash 0$. So $\phi(-\theta \cdot \xi, \xi) = 1$. Thus \eqref{eq-fou1} is equal to \eqref{eq-fou2} for $\xi\in \mbr^{3}$ and $\theta\in \mbs^2$ by the continuity of $\mcf_{(t, x)}(\kappa f)$. 
Taking  the inverse Fourier transform gives \eqref{eq-fou}.
\epf

Next, let $L^*$ be the $L^2$ adjoint of $L$ and $N  = L^*L$ be the normal operator. We recall from \cite{LOSU2} that   (for general dimension $n\geq 2$)
\beqq\label{eq-minnormal}
N f  = \int_{\mbr^{n+1}} K_N(t, x, t', x') f(t', x') dt'dx'
\eeqq
where the Schwartz kernel  
\beqq\label{eq-norker}
K_N(t, x, t', x') = \frac{\delta(t - t' - |x - x'|) + \delta(t - t' + |x - x'|)}{|x - x'|^{n - 1}}
\eeqq 
In particular, $N$ can be written as a Fourier multiplier 
\beqq\label{eq-normalfourier}
\begin{gathered}
N f(t, x) =   \int_{\mbr^{n+1}} e^{i(t, x)\cdot (\tau, \xi)} k(\tau, \xi) \hat f(\tau, \xi) d\tau d\xi
\end{gathered}
\eeqq
where 
\beqq\label{eq-ksym}
k(\tau, \xi) = C_n \frac{(|\xi|^2 - \tau^2)_+^{\frac{n - 3}{2}}}{|\xi|^{n-2}}, \quad C_n = 2\pi |\mbs^{n-2}|.
\eeqq
Here, for $t\in \mbr, \alpha \in \mbr$, we denote $t_+^\alpha$ the homogeneous distribution. We can write the normal operator as
\beq
N  f(t, x) =   \int e^{i(t - t', x - x')\cdot (\tau, \xi)} k(\tau, \xi)  f(t', x') d\tau d\xi dt'dx'
\eeq
This is a pseudo-differential operator away from $\tau^2 = |\xi|^2$ where the symbol $k(\tau, \xi)$ is singular. We will discuss later that the kernel belongs to the class of paired Lagrangian distributions, see also \cite{Wan1}.  However, it is not needed for Theorem \ref{thm-main1}. We recall the Sobolev estimate for $L, L^*$. 
\begin{lemma}[Theorem 1.2 of \cite{Wan1}]\label{prop-lrayest0}
For $n\geq 3, s\in \mbr$, the  light ray transform $L: H^s_{\comp}(\mbr^{n+1}) \rightarrow H_{\loc}^{s+ 1/2}(\mbr^n\times \mbs^{n-1}), s\in \mbr$ and its adjoint $L^*: H_{\comp}^{s}(\mbr^n\times \mbs^{n-1}) \rightarrow H^{s+1/2}_{\loc}(\mbr^{n+1}) $ are continuous. 
\end{lemma}

For $n = 3$, the analysis of $L$ and $N$ are simpler. Note that  
\beqq\label{eq-ksym1}
k(\tau, \xi) = 4\pi^2 \frac{\phi(\tau, \xi)}{|\xi|} 
\eeqq
It follows from \eqref{eq-normalfourier} that $N f  = N \phi(D) f$. 
Let $Q$ be defined by a Fourier multiplier 
\beq
\mcf(Q f)(\tau, \xi) = q(\tau, \xi) \hat f(\tau, \xi)
\eeq
where 
\beqq\label{eq-q}
q(\tau, \xi) = (4\pi^2)^{-1}   \phi(\tau, \xi)  |\xi|^{-1} 
\eeqq
Then we see that 
\beqq\label{eq-lrayinv}
QN \phi(D) f  = \phi(D) f
\eeqq
which means that $N$  is invertible on $\msh^s$.  
We need the mapping properties of $Q$ on Sobolev spaces. From \eqref{eq-q}, we can write $Q = \phi(D) Q_1$ where $Q_1$ is defined by Fourier multiplier as
\beq
Q_1 f= \mcf^{-1} ((2\pi)^{-2} |\xi|  \hat f(\xi)) 
\eeq
We see that $Q_1$ is an pseudo-differential operator of order $1$ on $\mbr^4$ (modulo a smoothing operator). We know that $\phi(D)$ is bounded on $H^s(\mbr^4), s\in \mbr$ thus 
\beq
Q: H_{\comp}^s(\mbr^{3+1})\rightarrow H_{\loc}^{s-1}(\mbr^{3+1})
\eeq
is bounded for $s\in \mbr.$

\section{Proof of Theorem \ref{thm-main1}}\label{sec-pfthm1}
We outline the proof of Theorem \ref{thm-main1}. We write \eqref{eq-X1} as 
\beq
X = \rho_T(\id - \id + (\id + T_1^{-1} K )^{-1}) T_1^{-1}  = L\kappa  + E
\eeq
where 
\beqq\label{eq-E}
\begin{gathered}
E = \rho_T(- \id + (\id + T_1^{-1} K )^{-1}) T_1^{-1}  
= \rho_T T_1^{-1}K (\id - T_1^{-1} K )^{-1}T_1^{-1}   
\end{gathered}
\eeqq
 To ``invert" $X$, we apply $L^*$ to $X$ to get $L^*X =   L^*L \kappa +  L^*E \kappa$. 
The operator $N = L^*L$ is invertible on $\msh^s$. In fact, it suffices to consider $L^*X$ on $\msh^s$ as follows. 

Consider the operator $T_1^{-1}$ defined in the proof of Theorem \ref{thm-exist}. For $\sigma$ depending only on $t$, we have
\beq
\begin{gathered}
T_1^{-1}f(t, x, \theta)  
 = \int_{0}^t \kappa(s) f(s, x+s\theta)ds
\end{gathered}
\eeq 
For any $t\in [0, T]$, let $\rho_t$ be the restriction operator to $\mcm_t = \{t\}\times \mbr^3$. We see that $\rho_t T_1^{-1}$ is also a  light ray transform with weight $\kappa$. The argument in Proposition \ref{prop-cut} implies that 
\beq
\rho_t T_1^{-1} f = \rho_t T_1^{-1}\phi(D) \kappa f.
\eeq
Because it holds for all $t$,  we have $X f = L \phi(D) \kappa f + E \phi(D)\kappa f$ 
which further gives 
\beqq\label{eq-XX}
\phi(D)L^*X f =  \phi(D) N \phi(D) \kappa f +    \phi(D)L^*E \phi(D)\kappa f
\eeqq
Now we can apply $Q$ and get 
\beqq\label{eq-QX}
Q\phi(D) L^*X f =  \phi(D) \kappa f + Q \phi(D) L^*E \phi(D)\kappa f
\eeqq
Recall that $\kappa$ is non-vanishing. It suffices to consider the right hand side of \eqref{eq-QX} as acting on functions in $\msh^s$, that is to consider an operator 
\beq
 \id + Q \kappa \phi(D) L^*E
\eeq
on $\msh^s$. The main part of the proof is to show that $Q \kappa \phi(D) L^*E : \msh^2\rightarrow \msh^2$ is compact. We know that $\kappa \phi(D)$ is bounded on $H^2(\mcm)$ and $Q$ is a pseudo-differential operator of order $1$. Thus, it suffices to show that $\p_t (L^*E), \p_{x_i} (L^*E): H^2 \rightarrow H^2, i = 1, 2, 3$ is compact which we prove below. 

We write $E$ in \eqref{eq-E} as
\beq
\begin{gathered}
E  = \rho_T T_1^{-1}K T_1^{-1} (\id - T_1^{-1} K )^{-1} 
\end{gathered}
\eeq
so that 
\beqq\label{eq-LE}
L^*E  = L^*  \rho_T T_1^{-1}  K T_1^{-1} (\id - T_1^{-1} K )^{-1}  
\eeqq
We observe that $\rho_T T_1^{-1}$ is almost $L$, however it is acting on $\mcm \times \mbs^2$ instead of $\mcm$. As in \cite{StUh}, we let $J$ be defined as $Jf(z, \theta) = f(z)$. Then $\rho_T T_1^{-1} J = L.$ Also, we note that $J$ is bounded $L^2(\mcm)\rightarrow L^2(\mcm\times \mbs^2)$ for example. To show the compactness, we adopt the idea in \cite{StUh} to decompose the operator \eqref{eq-LE}. 

For $k\in C^6(\mcm\times \mbs_{\theta}^2\times \mbs_{\theta'}^2)$, we write 
\beqq\label{eq-kaexpan}
k(z, \theta, \theta') = \sum_{j = 1}^\infty \Theta_j(\theta) k_j(z, \theta')
\eeqq
where $\Theta_j$ are spherical harmonics on $\mbs^2$ and $k_j$ are the corresponding Fourier coefficients. With our regularity assumptions, the series converges, see page 121 of \cite{StUh}.  
We remark that it suffices to impose weaker regularity assumptions on $k$ to get such expansion. But this is not our concern.  

Now we let $K_j, B_j$ be defined as
\beq
\begin{gathered}
K_j f(t, x, \theta) = \int \Theta_j(\theta) k_j(t, x, \theta') f(t, x, \theta') d\theta',  \\
B_j f(t, x) =  \int k_j(t, x, \theta) \kappa(t, x, s, \theta) f(s, x + s\theta, \theta) ds d\theta.
\end{gathered}
\eeq 
Using \eqref{eq-LE}, we get
\beqq\label{eq-LE1}
\begin{gathered}
L^*E = \sum_{j = 1}^\infty L^*  \rho_T T_1^{-1}  K_j T_1^{-1} (\id - T_1^{-1} K )^{-1}J  
= \sum_{j = 1}^\infty \underbrace{L^*  \rho_T T_1^{-1}  \Theta_j J}_{P_j^1} \underbrace{B_j   (\id - T_1^{-1} K )^{-1}J}_{P_j^2}
\end{gathered}
\eeqq
We analyze $P_j^1$ and $P_j^2$ separately.   
First,  we find the kernel of $P_j^1 = L^*L \Theta_j J$. We use the expression of $L^*$ in \cite{LOSU2} (see also Proposition \ref{prop-normal} in Section \ref{sec-lraywei} later) to get
\beq
\begin{gathered}
L^*L \Theta_j J f(t, x) =  \int_{\mbs^2} L\Theta_j J f(x - t\theta, \theta) d\theta  
 =  \int_{\mbs^2} \int_{\mbr} \Theta_j(\theta)  f(s, x + s\theta - t\theta) ds d\theta 
\end{gathered}
\eeq
Here, we split the integral in $s$ to $s\geq t$ and $s\leq t$. For $s\geq t$, we let $r = s- t$ and make a change of variable $y = x + r\theta$. The $s\leq t$ part can be treated similarly. We get 
\beq
\begin{gathered}
L^*L \Theta_j J f(t, x)  =  \int_{\mbr^3} [ \Theta_j(\frac{y-x}{|y-x|}) f(t + |x - y|, y) 
+  \Theta_j(\frac{x-y}{|x- y|}) f(t - |x - y|, y)] |x - y|^{-2}dy  
\end{gathered}
\eeq
Note that $L^*L\Theta_j f$ is compactly supported in $\mcm$ because $f$ is. 
Again, we write $f$ in Fourier series in $t$ as in Lemma \ref{lm-cpt} and use the notations there. We look at   
\beqq\label{eq-gn1}
\begin{gathered}
L^*L \Theta_j J g_n (t, x) =\int_{\mbr^3} \Theta_j(\frac{y-x}{|y-x|}) e^{i2 \pi n \frac{(t - |y - x|)}{T}} f_n(y) |x - y|^{-2}dy \\
+  \int_{\mbr^3}  \Theta_j(\frac{x-y}{|x- y|})e^{i2 \pi n\frac{ (t + |y - x|)}{T}} f_n(y) |x - y|^{-2}dy 
\end{gathered}
\eeqq
From Proposition \ref{prop-sinest1} (ii), we deduce that for fixed $t$ 
\beq
\|L^*L \Theta_j J g_n (t, \cdot)\|_{H^3(\mbr^3)}\leq C\|\Theta_j\|_{H^1} n^2\|f_n\|_{H^2(\mbr^3)}.
\eeq
thus
\beq
\|L^*L \Theta_j J f(t, \cdot) \|_{H^3(\mbr^3)} \leq C\|\Theta_j\|_{H^1}  \|f\|_{H^2(\mcm)}
\eeq
which implies that $L^*L \Theta_j J f \in L^2([0, T], H^3(\mbr^3))$. We see that $\p_{x^i} P_j^1, i = 1, 2, 3$ is bounded on $H^2(\mcm)$. We need to consider $\p_t P_j^1$ on $H^2$ for which we need three $t$ derivatives of $P_j^1$. For this purpose, note that each $t$ derivative of \eqref{eq-gn1} results in a factor $n$. So we need  $f\in H^5([0, T], H^2(\mbr^3))$ to get 
\beq
\|L^*L \Theta_j J f \|_{H^3(\mcm)} \leq C\|\Theta_j\|_{H^1}  \|f\|_{H^5([0, T], H^2(\mbr^3))}.
\eeq 
Thus $\p_t P_j^1, \p_x P_j^1$ are bounded from $H^5([0, T], H^2(\mbr^3))$ to $H^2(\mcm).$ 
This suggests us prove the following mapping property for $P_j^2.$ 
\begin{lemma}\label{lm-p2}
 $P_j^2: H^2(\mcm) \rightarrow H^5([0, T], H^2(\mbr^3))$ is compact.
\end{lemma}
 \bpf
We decompose $P_j^2$ as follows
\beqq\label{eq-p2}
P^2_j = B_jJ + B_j (\id - T_1^{-1} K )^{-1} K T_1^{-1}J 
\eeqq
We first show that $K T_1^{-1}J$ is compact from $H^2(\mcm)$ to $H^2(\mcm\times \mbr^2)$. We compute 
\beq
\begin{gathered}
K T_1^{-1} J f(t, x, \theta) = \int_{\mbs^{n-1}}k(t, x, \theta, \theta') \int_{0}^t  e^{- \int_s^t \sigma(\tilde s, x + \tilde s\theta', \theta')d\tilde s} f(s, x+s\theta')ds d\theta'
\end{gathered}
\eeq
Here, we included a general weight in $T_1^{-1}$ instead of $\kappa(s)$ because we need it later for Section \ref{sec-pf2}. 
Set $y = x + s\theta'$. We get $s = |y - x|, \theta' = (y - x)/|y - x|$ and 
\beq
\begin{gathered}
K T_1^{-1} J f(t, x, \theta) = \int_{\mcv} \frac{k(t, x, \theta, \frac{y-x}{|y-x|}) \kappa(t, x, |y-x|, \frac{y-x}{|y-x|})}{|y - x|^{n-1}} f(|y-x|, y)dy
\end{gathered}
\eeq
We use the Fourier  series of $f$ in $t$   as in Lemma \ref{lm-cpt} and the notations there to get 
\beq
\begin{gathered}
K T_1^{-1} J g_n (t, x, \theta) = \int_{\mcv} \frac{\alpha(t, x, y, \theta)}{|y - x|^{n-1}} e^{i2 \pi n \frac{|y - x|}{T}} f_n(y)dy 
\end{gathered}
\eeq
For fixed $t, \theta$,  it follows from Proposition \ref{prop-sinest1} (i) that  
\beq
\|K T_1^{-1} J g_n\|_{H^1} \leq C n^2 \|f_n\|_{L^2(\mbr^n\times \mbs^{n-1})}
\eeq
Summing up in $n$, we get for fixed $t, \theta$
\beq
\|K T_1^{-1} J f(t, \cdot, \theta)\|^2_{H^3(\mbr^3)} \leq C \sum_{n=-\infty}^\infty n^4 \|f_n\|^2_{L^2(\mbr^n\times \mbs^{n-1})} \leq C \|f\|^2_{H^2(\mcm)}.
\eeq
By considering $\p_{t, \theta}^\beta (K T_1^{-1} J f)$ for $|\beta|\leq 3$, we see that $KT_1^{-1}J$ is bounded from $H^2(\mcm)$ to $H^3(\mcm\times \mbs^2)$ thus compact to $H^2(\mcm\times \mbs^2)$.  We know that $(\id - T_1^{-1} K )^{-1}$ is bounded on $H^2$. It remains to consider $B_j$ on $H^2(\mcm\times \mbs^2)$ which  is
\beq
\begin{gathered}
B_j f(t, x) = \int k_j(t, x, \theta) \kappa(t, x, s, \theta) f(s, x + s\theta, \theta) ds d\theta\\
 =  \int \frac{k_j(t, x, \frac{y - x}{|y-x|}) \kappa(t, x, |x - y|, \frac{y - x}{|y-x|})}{|x - y|^2} f(|x - y|, y, \frac{y - x}{|y-x|}) dy
 \end{gathered}
\eeq
By considering the $t$ derivative which hits $k_j, \kappa$, we can show that $B_j$ is bounded from $H^2(\mcm\times \mbs^2)$ to $H^5([0, T]\times H^2(\mbr^3))$ using Proposition \ref{prop-sinest1}. The bound depends on $\|k_j\|_{C^5}$. It is very important that the $t$ derivatives do not produce $n$ factors as in the analysis for $P_j^1$. This is why we can gain regularity in $t$.  We thus proved that the second term of \eqref{eq-p2} is compact. 

Next, consider $B_j J$ in \eqref{eq-p2} which is  
\beq
B_j Jf(t, x) =   \int \frac{k_j(t, x, \frac{y-x}{|y-x|}) \kappa(t, x, |y-x|, \frac{y-x}{|y-x|})}{|y - x|^{2}} f(|y-x|, y)dy
\eeq
We use Proposition \ref{prop-sinest} and consider the $t$ derivative of $B_j J$ as above to conclude that $B_j J$ is bounded from $H^2(\mcm)$ to $H^6([0, T], H^3(\mbr^3))$, hence compact from $H^2(\mcm)$ to $H^5([0, T], H^2(\mbr^3))$. This shows the compactness of $P_2^j$  with a bound $C\|k_j\|_{C^6}$. This completes the proof of the lemma. 
\epf

Using these estimates, we see that $P_j^1, P_j^2$ are bounded on $H^2(\mcm)$ with a bound $\|\kappa_j\|_{C^6}\|\Theta_j\|_{H^1}$. Now consider derivatives of $k$ in $t, x, \theta$ variables in the series \eqref{eq-kaexpan}.  Using the argument in the end of page 121 of \cite{StUh}, we see that the series 
\beq
\sum_{j=1}^\infty \|\kappa_j\|_{C^6}\|\Theta_j\|_{H^1} <\infty
\eeq
with the regularity assumption of $k$.  Thus the series in \eqref{eq-LE} converges uniformly. This shows that $L^*E$ is compact on $H^2. $

\bpf[Completion of the proof of Theorem \ref{thm-main1}]
We use \eqref{eq-XX}
\beq
 L^*X f =  \phi(D) L^*L  \phi(D) \kappa f +   L^*E \phi(D)\kappa f 
\eeq
and apply $Q$ to get 
\beq
QL^*X f =    \phi(D) \kappa f + Q \kappa^{-1}L^*E \phi(D)\kappa f
\eeq
We recall that the function $\kappa$ depends on $\sigma.$ For $\la\in \mbc$ in a neighborhood of $[0, 1]$, we replace the scattering kernel $k$ by $\la k$ and denote the operator $K$ by $K(\la)$. Then we see that the operator in \eqref{eq-E}, now denoted by $E(\la)$ 
\beqq\label{eq-E1}
\begin{gathered}
E(\la)  
= \rho_T T_1^{-1}K(\la) (\id - T_1^{-1} K(\la) )^{-1}T_1^{-1}   
\end{gathered}
\eeqq
is a family of operators meromorphic in $\la.$ Suppose $E(\la)$ is holomorphic in $\mcu \subset \mbc$. For $\la = 0$, we know that $\id + \mce(0)$ is invertible. Thus by the analytic Fredholm theorem \cite[Theorem VI.14]{ReSi1}, we conclude that $\id + \mce(\la)$ with $\mce(\la) =  Q L^*E(\la)$ is invertible for $\la$ in $\mcu\backslash \mcs$ where $\mcs$ is a discrete set.  This implies that $Q L^*X$ is invertible for an open dense subset of $k$. Also, we obtain the stability estimate 
\beq
\|\phi(D) \kappa f\|_{H^2(\mcm)} \leq \|L^*X f\|_{H^3(\mcm)} \leq C\|Xf\|_{H^{5/2}(\mcc)}
\eeq
using the estimate of $L^*$ and $Q$ in Section \ref{sec-lray}.

Finally, for $f\in H^2_{\comp}(\mcm)$, we know that $Xf\in H^{5/2}(\mcc)$. If $X f = 0$, we get $\phi(D) \kappa f = 0$. By taking Fourier transform, we see that $\mcf(\kappa f)(\zeta) = 0$ for $\zeta\in \Gamma^{sp}$. But $\kappa f$ is compactly supported so $\mcf(\kappa f)(\zeta)$ is analytic in $\zeta$. We conclude that $\kappa f  = 0$ so $f = 0.$ This proves the uniqueness. 
\epf

\section{The light ray transform with weights}\label{sec-lraywei}
The rest of the paper is devoted to the proof of Theorem \ref{thm-main2}. It is worth pointing out the differences of the proof to Theorem \ref{thm-main1}.  
Following the notations in Section \ref{sec-direct}, we can represent $X$ as a perturbation of a weighted light ray transform $L_\kappa$ but it cannot be reduced to the Minkowski light ray transform when $\sigma$ depends on the $x$ variable. There are some results on the injectivity of the weighted light ray transform for analytic weights, see \cite{Ste1}. But the description of the kernel is unclear so the argument for Theorem \ref{thm-main1}, especially Proposition \ref{prop-cut} does not seem to work. 

The way we prove Theorem \ref{thm-main2} is to use the constraint of the Cauchy problem to obtain some stability estimate for the weighted light ray transform. Roughly speaking, for the Cauchy problem \eqref{eq-cauchy}, we can find a parametrix and  represent the solution as $f = E(f_1, f_2)$. We then show that the composition $L\chi_0 E$ can be microlocally inverted with a compact remainder. Then the analytic Fredholm argument can be applied. We remark that for the light ray transform without weights, similar stability estimates have been studied in \cite{VaWa, Wan1}. 

The desired stability estimate relies on the microlocal structure of the light ray transform which we study in this section. 
Let $\kappa(t, x, \theta)\in C^\infty(\mbr\times \mbr^3\times \mbs^2)$ be real valued with compact support in $t, x$ variables. We consider a weighted light ray transform 
\beqq\label{eq-lrayw}
L_\kappa f(x, \theta) =  \int_{-\infty}^\infty \kappa(s, x+ s\theta, \theta) f(s, x+s\theta)ds
\eeqq
for $f\in C_0^\infty(\mbr^4)$. 
Note that  \eqref{eq-T1} 
is of the form \eqref{eq-lrayw}. We first compute the Schwartz kernel of the normal operator.
\begin{prop}\label{prop-weiker}
Let $N_\kappa  = L_\kappa^*L_\kappa$. The Schwartz kernel of $N_\kappa$ as a distribution on $\mbr^4\times \mbr^4$ is given by
\beqq\label{eq-lrayker}
K(t, x, s, y) = \frac{\kappa(t, x, \frac{y-x}{|y - x|}) \kappa(s, y, \frac{x - y}{|x - y|})}{|y - x|^2} (\delta(t - s - |y - x|) + \delta(t - s + |y - x|)).
\eeqq
\end{prop}
\bpf
We start by computing $L^*_\kappa$. Let $g\in C_0^\infty(\mbr^3\times \mbs^2)$. We have
\beq
\begin{gathered}
\langle L_\kappa f, g\rangle = \int_{\mbs^2}  \int_{\mbr^3} \int_\mbr \kappa(s, x + s\theta, \theta) f(s, x + s\theta) g(x, \theta) ds dx d\theta\\
 = \int_{\mbr^3} \int_\mbr  \int_{\mbs^2} \kappa(s, y, \theta) f(s, y) g(y - s\theta, \theta)  d\theta ds dy 
\end{gathered}
\eeq 
where we made a change of variable via  $x + s\theta = y$. Thus, we get 
\beq
L_\kappa^* g(s, y) = \int_{\mbs^2} \kappa(s, y, \theta)  g(y - s\theta, \theta) d\theta
\eeq
Next, we compute 
\beqq\label{eq-int1}
\begin{gathered}
L_\kappa^*L_\kappa f(t, x) = \int_{\mbs^2} \kappa(t, x, \theta) L_\kappa f(x - t\theta, \theta) d\theta\\
 = \int_{\mbs^2} \int_\mbr \kappa(t, x, \theta)  \kappa(s, x - t\theta + s\theta, \theta)f(s, x - t\theta + s\theta)  ds d\theta 
\end{gathered}
\eeqq
Then we split the integral to $s\geq t$ and $s\leq t.$ For $s \geq t,$ we let $r = s - t\geq 0$ and use polar coordinates for $\mbr^3$ centered at $x$: $y =x + r\theta$.  We see that $\theta = (y - x)/|y - x|$ and $r = |y - x|$. In this case, the integral \eqref{eq-int1} denoted by $I_+$ below is 
\beq
\begin{gathered}
I_+ (t, x)  = \int_{\mbs^2} \int_0^\infty \kappa(t, x, \theta)  \kappa(r+ t, x + r\theta, \theta)f(r+ t, x+ r\theta)  dr d\theta \\
 = \int_{\mbr^3} \frac{\kappa(t, x, \frac{y - x}{|y-x|})  \kappa(t + |y-x|, y, \frac{y - x}{|y-x|})}{|y - x|^2}f(t+ |y-x|, y) dy\\
  = \int_{\mbr^4} \frac{\kappa(t, x, \frac{y - x}{|y-x|})  \kappa(s, y, \frac{y - x}{|y-x|})}{|y - x|^2} \delta(s - t - |y-x|) f(s, y) ds dy
\end{gathered}
\eeq

For $s\leq t$ in \eqref{eq-int1}, we follow the same procedure. Let $r = t - s\geq 0$ and we change $\theta$ to $-\theta$ in the integral. Then we use polar coordinate $y = x + r\theta$. If we denote the integral of \eqref{eq-int1} by $I_-$, we get 
\beq
\begin{gathered}
I_- (t, x)  = \int_{\mbs^2} \int_0^\infty \kappa(t, x, -\theta)  \kappa(s, x + r\theta, -\theta)f(s, x + r\theta)  dr d\theta \\
 = \int_{\mbr^3} \frac{\kappa(t, x, \frac{x - y}{|x-y|})  \kappa(t - |x-y|, y, \frac{x - y}{|x-y|})}{|y - x|^2}f(t - |y-x|, y) dy\\
  = \int_{\mbr^4} \frac{\kappa(t, x, \frac{x - y}{|x-y|})  \kappa(s, y, \frac{x - y}{|x-y|})}{|y - x|^2} \delta(s - t + |x-y|) f(s, y) ds dy
\end{gathered}
\eeq

Adding up $I_+$ and $I_+$, we obtain the kernel $K$. 
\epf
   
Next, we will show that the Schwartz kernel of $N_\kappa$ is a paired Lagrangian distribution from which we will derive Sobolev estimate for $N_\kappa$ and $L_\kappa$. For globally hyperbolic Lorentzian manifolds without conjugate points,  this was recently proved in \cite{Wan2} for the light ray transform without weight using the calculation of the Schwartz kernel of the normal operator. The method works the same for  the weighted light ray transform given Proposition \ref{prop-weiker}.  
 
We briefly recall the notion of paired Lagrangian distribution, see e.g.\ \cite{DUV}. Let $\mcx$ be a $C^\infty$ manifold of dimension $n$. 
Let $\La_0, \La_1$ be conic Lagrangian submanifolds of $T^*(\mcx\times \mcx)\backslash 0$.  Suppose that $\La_1$ intersects $\La_0$ cleanly at a co-dimension $k$, $1\leq k\leq 2n-1$ submanifold $\Sigma = \La_0\cap \La_1$, namely $T_p(\La_0\cap \La_1) = T_p(\La_0)\cap T_p(\La), \forall p\in \Sigma.$  It is known that all such intersecting pairs $(\La_0, \La_1)$ are locally symplectic diffeomorphic to each other. It suffices to consider the following model problem. Let $\tilde \mcx = \mbr^n = \mbr^k\times \mbr^{n-k}, 1\leq k\leq n-1$, and use coordinates $x = (x', x''), x'\in \mbr^k, x''\in \mbr^{n-k}$. Let $\tilde \La_0 = \{(x, \xi, x, -\xi)\in T^*(\tilde \mcx\times \tilde \mcx)\backslash 0 : \xi\neq 0\}$ be the punctured conormal bundle of $\diag$ in $T^*(\tilde \mcx \times \tilde \mcx)$, and 
\beq
\tilde \La_1 = \{(x, \xi, y, \eta)\in T^*(\tilde \mcx\times \tilde \mcx)\backslash 0: x'' = y'', \xi' = \eta' = 0, \xi'' = \eta'' \neq 0\} 
\eeq
which is the punctured conormal bundle to $\{(x, y)\in \tilde \mcx\times \tilde \mcx: x'' = y''\}$.  The two Lagrangians intersect cleanly at $\tilde \Sigma = \{(x, \xi, y, \eta)\in T^*(\tilde \mcx\times \tilde \mcx)\backslash 0 : x'' = y'', \xi'' = \eta'', x' = y', \xi' = \eta' = 0\}$ which is of co-dimension $k.$ For this model pair, the paired Lagrangian distribution $I^{p, l}(\mbr^n\times \mbr^n; \tilde \La_0, \tilde \La_1)$ consists of oscillatory integrals 
\beqq\label{eq-upair}
u(x, y) = \int e^{i[(x' - y' - s)\cdot \eta' + (x'' - y'')\cdot \eta'' + s\cdot \sigma]}a(s, x, y, \eta, \sigma) d\eta d\sigma ds
\eeqq
where $a$ is a product type symbol  which is a $C^\infty$ function and satisfies 
\beqq\label{eq-symord}
|\p_\eta^\alpha \p_\sigma^\beta \p_s^\theta \p^\gamma_x \p^\delta_y a(s, x, y, \eta, \sigma)| \leq C (1 + |\eta|)^{p + k/2 -|\alpha|}(1 + |\sigma|)^{l   - k/2 - |\beta|}
\eeqq
for multi-indices $\alpha, \beta, \theta, \gamma, \delta$ over each compact set $\mck$ of $\mbr^n\times \mbr^n \times \mbr^k.$ The constant $C$ depends on the indices and $\mck.$ The set of product type symbols is denoted by $S^{p, l}(\mbr^n\times \mbr^n; \mbr^n; \mbr^k)$. 
We use the notation $I^{p, l}(\mbr^n\times \mbr^n; \tilde \La_0, \tilde \La_1)$ to denote the space of operators $A: \mce'(\mbr^n; \Omega^\ha_{\mbr^n}) \rightarrow \mcd'(\mbr^n; \Omega^\ha_{\mbr^n})$ where $\Omega^\ha_{\mbr^n}$ denotes the line bundle of half-densities on $\mbr^n$,  whose Schwartz kernel $K_A$ is a paired Lagrangian distribution with values in $\Omega_{\mbr^n\times \mbr^n}^\ha$. Away from $\tilde \Sigma = \tilde \La_0\cap \tilde \La_1$, $u \in I^{p+l}(\mbr^n\times \mbr^n; \tilde \La_0)$ and  $u \in I^{p}(\mbr^n\times \mbr^n; \tilde \La_1)$ using H\"ormander's notion of Lagrangian distributions, see \cite[Section 25.1]{Ho4}.

 There is an equivalent description of the paired Lagrangian distribution \eqref{eq-upair} introduced in \cite[Section 5]{DUV} that is convenient for our purpose. Modulo $C_0^\infty(\mbr^n\times \mbr^n)$, \eqref{eq-upair} can be written as 
 \beqq\label{eq-upair1}
u(x, y) = \int e^{i[(x' - y')\cdot \eta' + (x'' - y'')\cdot \eta'']}b(x, y, \eta) d\eta 
\eeqq
where $b$ satisfies the following estimates. First, in the region $|\eta'|\leq C |\eta''|, |\eta''|\geq 1$, $b$ satisfies
\beq
|(Qb)(x, y, \eta)|\leq C\langle \eta''\rangle^{p+k/2} \langle \eta'\rangle^{l - k/2}
\eeq
for all $Q$ which is a finite product of differential operators of the form $D_{\eta'}, \eta'_j D_{\eta'_m}, $ $ \eta''_j D_{\eta_m''}$. Second, in the region $|\eta''|\leq C |\eta'|, |\eta'|\geq 1$, $b$ satisfies the standard regularity estimate 
\beq
|(Qb)(x, y, \eta)|\leq C\langle \eta'\rangle^{p+ l}
\eeq
for all $Q$ which is a finite product of differential operators of the form $\eta_j'D_{\eta'_m},$ $\eta'_j D_{\eta''_m}.$ We refer the readers to \cite[Section 5]{DUV} for the argument of the equivalence of \eqref{eq-upair} and \eqref{eq-upair1}, and a description of the principal symbols \cite[Lemma 5.3]{DUV}. 

We will show that $N_\kappa$ is a paired Lagrangian distribution associated with two Lagrangians $\La_0, \La_1$ described as follows. First, let 
 \beqq\label{eq-lag1}
 \begin{gathered}
 \La_0 = \{(t, x, \tau, \xi; t', x', \tau', \xi')\in T^*\mbr^{n+1}\backslash 0\times T^*\mbr^{n+1}\backslash 0: \\
 t = t', x = x', \tau = -\tau', \xi = -\xi'\}
 \end{gathered}
 \eeqq
which is the punctured conormal bundle of the diagonal in $\mbr^{n+1}\times \mbr^{n+1}$ and 
 \beqq\label{eq-lag2}
 \begin{gathered}
 \La_1 = \{(t, x, \tau, \xi; t', x', \tau', \xi')\in T^*\mbr^{n+1}\backslash 0\times T^*\mbr^{n+1}\backslash 0: \\
 x = x' + (t - t')\xi/|\xi|, \tau = \pm|\xi|,  \tau' = -\tau,  \xi' = -\xi\}.
  \end{gathered}
 \eeqq 
The two Lagrangians intersect cleanly at 
\beqq\label{eq-sigma}
\begin{gathered}
\Sigma = \{(t, x, \tau, \xi; t', x', \tau', \xi')\in T^*\mbr^{n+1}\backslash 0\times T^*\mbr^{n+1}\backslash 0: t = t', x = x', \\
\tau = -\tau', \xi = -\xi', \tau^2 = |\xi|^2\}
\end{gathered}
\eeqq
In fact, $\La_1$ is the flow out of $\Sigma$ under the Hamilton vector field $H_f$ of $f(\tau, \xi) = \ha(\tau^2 - |\xi|^2).$  Indeed, we have 
\beq
H_f = \tau \frac{\p}{\p t} + \sum_{i = 1}^3 \xi_i \frac{\p}{\p x^i}
\eeq
Let $\gamma(s) = (t(s), x(s), \tau(s), \xi(s))$ be a null bi-characteristic which satisfy 
\beq
\begin{gathered}
\dot t(s) = \tau, \quad \dot x_i(s) = \xi_i, \quad \dot \tau(s) = 0, \quad \dot \xi_i(s) = 0\quad s\geq 0\\
 t(0) = t', \quad x(0) = x', \quad \tau(0) = \tau', \quad \xi(0) = \xi'
\end{gathered}
\eeq
with $f(\tau', \xi') = 0.$ We solve that 
\beq
t(s) = t' +  s \tau', \quad x(s) = x' + s \xi', \quad \tau(s) = \tau', \quad \xi(s) = \xi'.  
\eeq
which up to a re-parametrization gives \eqref{eq-lag2}.

\begin{prop}\label{prop-normal} 
For the weighted light ray transform $L_\kappa$ in \eqref{eq-lrayw}, the Schwartz kernel of the normal operator $N_\kappa = L_\kappa^\ast L_\kappa$ belongs to $I^{-3/2, 1/2}(\mbr^{3+1}\times \mbr^{3+1}; \La_0, \La_1)$, in which $\La_0, \La_1$ are two cleanly intersection Lagrangians defined in \eqref{eq-lag1}, \eqref{eq-lag2}.  The principal symbol  of $N_\kappa$ on $\La_1\backslash \Sigma$ is non-vanishing if $\kappa$ is a positive smooth function.  
\end{prop}
\bpf
To see that the kernel  is a paired Lagrangian distribution, we will find its oscillatory integral representation. 
We start from 
\beq
N_\kappa f(t, x) = \int K(t, x, t', x') f(t', x') dt'dx'
\eeq
where $K$ is given by Proposition \ref{prop-weiker}. 
 Using the Fourier transform of the delta distribution, we have
\beq
\begin{gathered}
N_\kappa f(t, x) =C \int_{\mbr^3}\int_\mbr  \int_\mbr ( e^{i(t - t' - |x'-x|)\tau}   + e^{i(t - t' + |x' -x|)\tau} )\\
\cdot \frac{\kappa(t, x, \frac{x'-x}{|x' - x|}) \kappa(t', x', -\frac{x' - x}{|x' - x|})}{|x' - x|^2}  f(t', x')d\tau  dt'dx' \\
= \int_{\mbr^3}\int_\mbr  \int_\mbr e^{i(t - t')\tau} B(t, x, t', x', \tau)f(t', x')d\tau  dt' dx'
\end{gathered}
\eeq
where $C$ is a non-zero constant due to the Fourier transform which we do not keep track of, and 
\beq
\begin{gathered}
B(t, x, t', x', \tau) = (e^{i|x'-x|\tau} + e^{-i|x'-x|\tau}) \frac{\kappa(t, x, \frac{x'-x}{|x' - x|}) \kappa(t', x', -\frac{x' - x}{|x' - x|})}{|x' - x|^2}\\
 = (e^{i|z|\tau} +e^{-i|z|\tau}) \frac{\kappa(t, x, \frac{z}{|z|}) \kappa(t', x+z, -\frac{z}{|z|})}{|z|^2}
 \end{gathered}
\eeq
Here, we used $z = x'-x$.  We take Fourier transform in $z$ variable to get 
\beqq\label{eq-nka}
\begin{gathered}
N_\kappa f(t, x)  
 = C \int_{\mbr^3}\int_\mbr  \int_{\mbr^3} \int_\mbr e^{i(t - t')\tau} e^{i(x'-x)\xi} A(t, x, t', \tau, \xi)f(t', x')d\tau  d\xi dt' dx'   
\end{gathered}
\eeqq
where  
\beq
A(t, x, t', \tau, \xi) = \int_{\mbr^3} (e^{-iz' \xi- i|z'|\tau} + e^{-iz' \xi + i|z'|\tau})
   \frac{\kappa(t, x, \frac{z'}{|z'|}) \kappa(t', x + z', -\frac{z'}{|z'|})}{|z'|^2}  dz'
\eeq
Let $z' = r\theta$ in polar coordinate. We get 
\beqq\label{eq-tA}
\begin{gathered}
A(t, x, t', \tau, \xi) =  \int_{\mbs^2}\int_{0}^\infty (e^{-ir\theta \cdot \xi- ir\tau} + e^{-ir\theta \cdot \xi + ir\tau}) 
   \kappa(t, x, \theta) \kappa(t', x + r\theta, -\theta) dr d\theta
   \end{gathered}
\eeqq
Without loss of generality, we assume that $\kappa(t', x+ r\theta, -\theta)$ is compactly supported in $r$. Using Fourier transform in $r$ variable and integration by parts, we see that 
\beqq\label{eq-nka1}
\begin{gathered}
A(t, x, t', \tau, \xi)      =  \int_{\mbs^2}  F(t, t', x, \theta\cdot \xi + \tau, \theta) + F(t, t', x, \theta\cdot \xi - \tau, -\theta) d\theta
   \end{gathered}
\eeqq
where $F(t, t', x, \sigma, \theta)$ is smooth and compactly supported in $t, t', x$ variables. Furthermore, for $|\sigma|$ large and $k\geq 0,$ we have
\beqq\label{eq-FF}
|\p_\sigma^k F|\leq C_k |\sigma|^{-k-1} 
\eeqq
where $C_k$ depends on $k.$ 
We  compute the integral in \eqref{eq-nka1} in $\theta$.  It suffices to consider $\xi = |\xi|(0, 0, 1)$ then $\theta \cdot \xi = |\xi|\theta_3$. We let $v = |\xi|\theta_3$, and use $w^2 + \theta_3^2 = 1$ to get 
\beqq\label{eq-a1}
\begin{gathered}
A(t, x, t', \tau, \xi)     =  \int_{-1}^1 \int_{-1}^1  (F(t, t', x, \theta_3 |\xi| + \tau, \theta)  
+ F(t, t', x, \theta_3 |\xi| - \tau, -\theta)) (1 - \theta_3^2)^\ha d\theta_3 dw
   \end{gathered}
\eeqq

  Now that we computed the amplitude function in \eqref{eq-nka},  we are ready to show that the kernel of $N_\kappa$ is a paired Lagrangian distribution.  We will transform \eqref{eq-nka} to the model form \eqref{eq-upair1}.  We make a symplectic change of variables   on $T^*\mbr^{n+1}\backslash 0$
\beq
\tilde x = x + t\xi/|\xi|, \quad \tilde t = t, \quad s = \tau - |\xi|, \quad  \xi = \xi.
\eeq 
We can choose an Fourier integral operator with symbol of order $0$ which quantizes the symplectic change of variable to transform $K$ to
\beqq\label{eq-kker}
\begin{gathered}
K(\tilde t, \tilde x, t', x')   =  \int_{\mbr^{n+1}} e^{i\tilde t s + i \tilde x\cdot \xi} A(t, x, t', s + |\xi|, \xi) ds d\xi 
 \end{gathered}
\eeqq 
modulo a smooth term. See also \cite{Wan1}. We claim that $A(t, x, t', s + |\xi|, \xi)$ satisfies the product type estimate with $p = -3/2, l = 1/2$. From \eqref{eq-a1}, we have
\beqq\label{eq-a2}
\begin{gathered}
A(t, x, t', s +|\xi|, \xi) 
 = \int_{\mbs^1} \int_{-1}^{1} F(t, t', x, (\alpha + 1) |\xi| + s , \alpha, w)  (1 - \alpha^2)^\ha  d\alpha dw\\
+  \int_{\mbs^1} \int_{-1}^{1} F(t, t', x, (\alpha -1)|\xi| - s, -\alpha, -w)) (1 - \alpha^2)^\ha d\alpha dw 
   \end{gathered}
\eeqq
First, for $|\xi| \leq C|s|, |s|\geq 1$,  we  use the estimate \eqref{eq-FF} to get  
\beq
|A(t, x, t', s+ |\xi|, \xi)| \leq   C|s|^{-1}
\eeq
One can verify the same estimate for $Qk$ where $Q$ is the finite product of differential operators of the form 
$sD_{s}, s D_{\xi_m}$. For $|s|\leq C|\xi|, |\xi|\geq 1$, we have 
\beq
|A(t, x, t', s + |\xi|, \xi)| \leq   C  |\xi|^{-1} 
\eeq
 and one can verify the estimate for $Qk$ where $Q$ is the finite product of differential operators of the form $D_s, s D_{s}, \xi_j D_{\xi_m}$. So $K$ is a paired Lagrangian distribution according to \eqref{eq-upair1}. 
 
 Finally, we compute  the principal symbol on $\La_1\backslash \La_0$. For this purpose, we can actually use the kernel representation in Proposition \ref{prop-weiker} and find that 
\beqq\label{eq-lrayker1}
\begin{gathered}
N_\kappa (t, x, t', x') = \int_{\mbr} \frac{\kappa(t, x, \frac{x'-x}{|x' - x|}) \kappa(t', x', \frac{x - x'}{|x - x'|})}{|t- t'|^2} 
\cdot (e^{i(t - t' - |x - x'|)\tau} + e^{i(t - t' + |x - x'|)\tau}) d\tau 
\end{gathered}
\eeqq 
which is valid for $t'\neq t.$ 
This gives an oscillatory representation of the Fourier integral operator on $\La_1$ with a real phase function, see \cite[Chapter VI]{Tre}. We see that the  symbol   is non-vanishing and positive because $\kappa$ is positive.  
\epf

We derive the Sobolev estimates for $L_\kappa$ and $L_\kappa^*$. 
\begin{prop}\label{prop-sobo}
The weighted light ray transform $L_\kappa$ in \eqref{eq-lrayw} is bounded from $H^s_{\comp}(\mbr^{3+1})$ to  $H_{\loc}^{s+ 1/2}(\mbr^3\times \mbs^{2})$. Its adjoint $L_\kappa^*$ is bounded from $H_{\comp}^{s}(\mbr^3\times \mbs^{2})$ to $H^{s+1/2}_{\loc}(\mbr^{3+1})$. 
\end{prop} 
\bpf
 
Because $\La_1$ is the flow out of $\La_0\cap \La_1$ see above Proposition \ref{prop-normal}, we can apply Proposition 5.6 of \cite{DUV} to get the estimates for $N_\kappa$ which further gives estimates for $L_\kappa, L_\kappa^*.$ 
\epf

\section{Analysis of the composition}\label{sec-comp}
We follow the approach in \cite{Wan2} to analyze the composition of the weighted light ray transform and the parametrix of the Cauchy problem. 
We briefly recall the parametrix of the Cauchy problem \eqref{eq-cauchy}. The operator $P(z, \p)$ in \eqref{eq-hyper} is strictly hyperbolic of multiplicity one with respect to the Cauchy hypersurfaces $\mcm_t$, see Definition 5.1.1 of \cite{Dui}. This means that all bi-characteristic curves of $P$ are transversal to $\mcm_t$ and for $(\bar z, \bar \zeta) \in T^*\mcm_t \backslash 0$
\beq
\mcp(\bar z, \zeta) = 0, \quad \zeta|_{T_{\bar z}\mcm_t} = \bar \zeta
\eeq
has exactly one solution.   It is convenient to use $D_j = -\imath \p_j, j = 0, 1, 2, \cdots, n$ in which $\imath^2 = -1.$ Also, it is convenient to work on a larger set. We let $\mcn = (-\tilde T, \tilde T)\times \mbr^3$ for $\tilde T > T$ and consider the Cauchy problem on $\mcn$
\beqq\label{eq-cauchy1}
\begin{gathered}
P(z, D)u(z) = 0, \text{ on } \mcn \\
u = f_1, D_t u = f_2 \text{ on } \mcn_0.
\end{gathered}
\eeqq
We use Duistermaat-H\"ormander's parametrix construction, see e.g.\ \cite{Dui}. 
The restriction operator $\rho_0: C^\infty(\mcn) \rightarrow C^\infty(\mcn_0)$ is an FIO  in $I^{1/4}(
\mcn, \mcn_0; C_0)$ with canonical relation  
\beqq
C_0 = \{(z, \zeta, \bar z, \bar \zeta) \in T^* \mcn \backslash 0 \times T^*\mcn_0 \backslash 0 : \bar z = z, \bar \zeta = \zeta|_{T_{\bar z}\mcn_0}\}
\eeqq
We consider the canonical relation $C_{wv}$ defined by
\beqq\label{eq-canowv}
\begin{gathered}
C_{wv} = \{ (w, \iota, \bar z, \bar \zeta) \in T^*\mcn \backslash 0 \times T^*\mcn_0 \backslash 0: \text{$(w, \iota)$ is on the bicharacteristic }\\
\text{strip through some $(\bar z, \zeta)$ such that } \bar \zeta = \zeta|_{T_{\bar z}\mcn_0} \text{ and } \mcp(\bar z, \zeta) = 0\}
\end{gathered}
\eeqq
The  next result is straight forward from Theorem 5.1.2 of \cite{Dui}.
\begin{prop}\label{prop-wvpara}
There exists $E_1 \in I^{-1/4}(\mcn, \mcn_0; C_{wv}), E_2 \in  I^{-5/4}(\mcn, \mcn_0; C_{wv})$ such that 
\beqq\label{eq-parawv}
\begin{gathered}
P(z, D)E_k \in C^\infty(\mcm), \quad k = 1, 2\\
\rho_0 E_1 - \id \in C^\infty(\mcn_0), \quad \rho_0 E_2 \in C^\infty(\mcn_0)\\
\rho_0 D_t E_1 \in C^\infty(\mcn_0), \quad  \rho_0 D_t E_2  - \id \in C^\infty(\mcn_0)
\end{gathered}
\eeqq
\end{prop}
Now we represent the solution of \eqref{eq-cauchy1} as $
u = E_1 f_1 + E_2 f_2
$ modulo a smooth term.  
It is natural to decompose $C_{wv}$ as the disjoint union of $C_{wv}^+$ and  $C_{wv}^-$ which are 
\beqq\label{eq-canowvpm}
\begin{gathered}
C^\pm_{wv} =  \{(w, \iota, \bar z, \bar \zeta) \in T^*\mcn \backslash 0 \times T^*\mcn_0 \backslash 0: \text{$\iota$ is future/past }\\
\text{pointing light-like and lies on the bicharacteristic strip  through}\\
\text{ some $(\bar z, \zeta)$ such that }\bar \zeta = \zeta|_{T_{\bar z}\mcn_0} \text{ and } \mcp(\bar z, \zeta) = 0\}
\end{gathered}
\eeqq
We can decompose (for $k = 1, 2$)
\beq
E_k = E_k^+ + E_k^-, \quad E_k^\pm \in I^{1- k - 1/4}(\mcn, \mcn_0; C_{wv}^\pm).
\eeq
We need the relation of the principal symbols of $E_1^\pm, E_2^\pm$. We remark that the Maslov bundle and the half density bundle can be trivialized because the Lagrangians involved allow global parametrization. We will not show these factors in the notations. 
\begin{lemma}[Lemma of \cite{Wan2}]\label{lm-wvsym}
Let $e_k^\pm, k = 1, 2$ be the principal symbol of $E_k^\pm$ on $\La^\pm = (C^{\pm}_{wv})'$ respectively. Suppose that the sub-principal symbol of $P(z, \p)$ is purely imaginary, in which case $P(z, \p)$  is of the form
\beqq\label{eq-hyper1}
P(z, \p) = \square + \sum_{j = 0}^n A_j(z)\p_j + B(z)
\eeqq
where $A_j(z)$ are real valued smooth functions. Then $e_k^\pm, k = 1, 2$ are  real valued and 
\beq
e_1^+ > 0, \quad e_2^+ >0, \quad e_1^- >0, \quad e_2^- <0.
\eeq
\end{lemma}

 Now we can outline the steps for inverting the composition $L_\kappa \chi_0 E$ where $E = E_k^\pm, k = 1, 2$ in Proposition \ref{prop-wvpara}.   The idea is to consider the normal operator $(L_\kappa E)^*(L_\kappa E)$ and fine-tune it so the operator is well-behaved.  Let $\chi$ be a smooth function supported in $(T, T')\times \mbr^3.$  
 We analyze $E^*\chi N _\kappa  \chi_0 E$ and show it is an elliptic pseudo-differential operator on $\mcn_0.$ There are three main ingredients.     
\begin{enumerate}[(i)]
\item As $\chi \cdot  \chi_0 = 0$, we know that $\chi N_\kappa \chi_0 \in I^{-3/2}(\mbr^{4}, \mbr^{4}; \La_1)$. The role of $\chi$ is to keep the kernel of $N_\kappa$ away from the diagonal $\La_0$ where the principal symbol is singular. 
\item Let $\La_\pm = (C_{wv}^\pm)'$. It was shown in \cite{Wan2} that $\La_1$ intersect $\La_\pm$ cleanly with excess one so the composition $\chi N_\kappa  \chi_0 E \in I^{\ast}(\mcn, \mcn_0; C_{wv})$ as  a result of Duistermaat-Guillemin's clean FIO calculus with the order $\ast$ to be determined.  
\item We can compose the operator in (ii) with $E^*$ by using clean FIO calculus again to conclude that $E^*\chi N_\kappa \chi_0 E \in \Psi^{\ast}(\mcn_0)$. The operator can be shown to be elliptic and a parametrix can be constructed.   
\end{enumerate}
 
%
%

We consider the compositions in (ii) and (iii). In fact, we will include some pseudo-differential operators such as those showed up in Section \ref{sec-cmb}. 
\begin{lemma}\label{lm-comp3}
Let $A(D)$ be a pseudo-differential operator on $\mcm$ of order $m.$ We assume that the principal symbol of $A(D)$ is non-vanishing on $\La_\pm$.  
Then the composition $\chi N_\kappa  \chi_0 A(D) E_k^\pm \in I^{-1/4 - k + m}(\mcn, \mcn_0; C^\pm_{wv})$ and the principal symbol is non-vanishing. 
\end{lemma}
\bpf 
Note that $A(D) E_k^\pm \in I^{1 - k + m}(\mcn, \mcn_0; C^\pm_{wv})$ with non-vanishing principal symbol. 
Because $\chi(t) \chi_0 (t) = 0$, we know  that $\chi N_\kappa \chi_0 \in I^{-3/2}(\mcn, \mcn; \La_1)$. It is proved in Lemma 6.1 of \cite{Wan2} that $\La_1$ intersects $\La_\pm$ cleanly. 
One can apply the clean FIO calculus \cite[Theorem 25.2.3]{Ho4} directly to see that $\chi N_\kappa  \chi_0 A(D) E_k^\pm \in I^{-3/2 + 1/4 + 1 - k + m}(\mcn, \mcn_0; C^\pm_{wv})$.  For $p = (t, x, \tau, \xi, y, \eta)\in \La^\pm$, let $C_p$ be the fiber over $p$ in $T^*\mcm\times T^*\mcm \times T^*\mcn_0$ which is connected and compact. Then the principal symbol of the composition at $p$ is given by 
\beqq\label{eq-fiber}
\int_{C_p} \sigma(\chi N_\kappa  \chi_0)(t, x, \tau, \xi, t', x', \tau', \xi')\sigma(A(D)E_k^\pm)( t', x', \tau', \xi', y, \eta)
\eeqq
where $\sigma(\chi N_\kappa  \chi_0), \sigma(A(D)E_k^\pm)$ denote the principal symbols of $\chi N_\kappa \chi_0, E_k^\pm$ respectively and the integration is over the fiber $C_p$, see \cite[Theorem 25.2.3]{Ho4}. With proper choice of the phase function (modulo the Maslov factor), both symbols are  real valued and non-vanishing on the fiber, see Lemma \ref{lm-wvsym} and the proof of Proposition \ref{prop-normal}. We see that the principal symbol of the composition is  non-vanishing.
\epf

\begin{lemma}\label{lm-comp4}
Let $A(D)$ be as in Lemma \ref{lm-comp3}. For $j, k = 1, 2$, we have
\begin{enumerate}
\item $E^{\pm, \ast}_j \chi N_\kappa  \chi_0 A(D) E_k^\pm  \in \Psi^{1 - j - k + m}(\mcn_0)$ are elliptic. 
\item $E^{+, \ast}_j \chi N_\kappa  \chi_0 A(D) E^-_k, E^{-, \ast}_j \chi N_\kappa  \chi_0 A(D)E^+_k$ are smoothing operators on $\mcn_0.$
\end{enumerate}
\end{lemma}
\bpf
 First, $E^{\pm, *}_j \in I^{-1/4 + 1-j}(\mcn_0, \mcn; C_{wv}^{\pm, -1})$ and $\chi N_\kappa  \chi_0 A(D)E_k^\pm \in I^{-1/4  -k + m}$ $(\mcn, \mcn_0; C_{wv}^\pm)$. Let $\La^\pm = (C_{wv}^\pm)'$ and $\La^{\pm, -1} = (C_{wv}^{\pm, -1})'$.  It is proved in Lemma 6.3 of \cite{Wan2} that  $\La^{\pm, -1}$ intersect $\La^\pm$  cleanly with excess one.   Now we can use the clean FIO calculus \cite[Theorem 25.2.3]{Ho4} to conclude that $E^{\pm, \ast}_j \chi N_\kappa  \chi_0 A(D)E_k^\pm$ $ \in \Psi^{1 - j - k + m}(\mcn)$. As both principal symbols of $E^{\pm, \ast}_j$ and $  \chi N_\kappa  \chi_0 A(D)E_k^\pm$ are real  and   non-vanishing (modulo the Maslov factor), the principal of the composition is the integration of the product of principal symbols so is also non-vanishing. This proves part (1).  
Part (2) can be seen from a wave front set analysis using e.g.\ \cite[Theorem 1.3.7]{Dui}. 
 \epf

\section{Proof of Theorem \ref{thm-main2}} \label{sec-pf2}
 We will prove a stronger version of Theorem \ref{thm-main2}. 
 \begin{theorem}\label{thm-main3} 
{\bf (Assumption)} Let $f$ be the solution of \eqref{eq-cauchy} on $\mcm$ with Cauchy data $f_1 \in H^{2}(\mcm_0), f_2\in H^{1}(\mcm_0), s\in \mbr$ supported in a compact set $\mcx$ of $\mcm_0$ such that $f$ is supported in $\mcv.$ Suppose that  
\begin{enumerate}
 \item the coefficients $A_j(z)$ in \eqref{eq-hyper} are real valued smooth functions. 
 \item $A(D)$ is a pseudo-differential operator as in Lemma \ref{lm-comp3} 
\item  When $\sigma = k = 0$, $L\chi_0 A(D)f = 0$  implies $f = 0.$ 
 \end{enumerate}
 Let $u$ be the solution of \eqref{eq-boltz} with zero initial condition and source $\chi_0A(D)f$. 
 
{\bf (Conclusion)} There exists an open dense set $\mcu$ of $C^\infty(\mcv\times \mbs^2)\times C^6(\mcv\times \mbs^2\times \mbs^2)$ such that for $(\sigma, k)\in \mcu$, 
 $f$ is uniquely determined by $u_T$. Furthermore, there exists $C >0$ such that  
\beq
\|u\|_{H^{2+m}(\mcm)} \leq C \|(f_1, f_2)\|_{H^{2+m}(\mcm_0)\times H^{1+m}(\mcm_0)}\leq C \|u_T\|_{H^{5/2}(\mcc)}
\eeq 
\end{theorem}

We make a few remarks. First, because the parametrix construction in Section \ref{sec-comp} is microlocal, it is convenient to work with smooth $\sigma$. However, it is likely to lower the regularity requirement if one is willing to keep track of the dependency of the constant $C$ on $\sigma$ throughout the arguments.   Next, about  assumption (3). This condition only involves the Minkowski light ray transform or the transport regime ($\sigma = k = 0$). Let's consider the example in Section \ref{sec-cmb} in which $A(D) = \frac{\p}{\p t} + B(z)$. When $\sigma = k = 0$, $B(z) = 0$ so $A(D) = \p_t.$ Note that $f$ satisfies \eqref{eq-bar}. If $L(\chi_0A(D)f) = 0,$ because $\chi_0A(D)f$ is compactly supported in $\mbr^4$  and say $\chi_0A(D)f\in L^2_{\comp}(\mcm)$, we know from the injectivity of Minkowski light ray transform that $\chi_0A(D)f = 0$ so $\p_t f = 0$ on the support of $\chi_0$ Using the Bardeen's equation \eqref{eq-bar}, we get 
\beq
\lap f + B_0(t) f = 0
\eeq
for $t\in [0, T]$ almost everywhere. Fixed $t$, using the fact that $f$ is compactly supported in $x$ and that $\lap$ has no $L^2$ eigenvalues, we conclude that $f(t, x) = 0$ on $\supp\chi_0\times \mbr^3.$ By solving the Cauchy problem backward, we conclude that $f_1 = f_2 = 0$ so $f = 0$ on $\mcm.$ Thus assumption (3) of Theorem \ref{thm-main3} is satisfied. 

 We outline the proof of Theorem \ref{thm-main3}. We start with \eqref{eq-X}
\beq
X  \chi_0 f = \rho(\id - \id + (\id + T_1^{-1} K )^{-1}) T_1^{-1} \chi_0 f  = L_\kappa  \chi_0 f  + E \chi_0 f
\eeq
where 
\beq
E = \rho(- \id + (\id + T_1^{-1} K )^{-1}) T_1^{-1} 
\eeq
Now we write  $f = E_1f_1 + E_2 f_2$  
where $E_1, E_2$ are the parametrix of the Cauchy problem. We consider 
\beqq\label{eq-Xnew}
\begin{gathered}
X  \chi_0 A f = L_\kappa    \chi_0  A E_1 f_1 +  L_\kappa  \chi_0 A E_2 f_2 
  + E  \chi_0 A ( E_1 f_1 + E_2  f_2 )
  \end{gathered}
\eeqq 
Our goal is to convert the right hand side to identity plus compact operators. Roughly speaking, we will apply $E^{\pm, *}_k \chi L_\kappa^*, k = 1, 2$   to \eqref{eq-Xnew} and manipulate to get $f_1, f_2$. As a result, we need to analyze operators of two types. First 
\beqq\label{eq-pp1}
P_1 = E^{\pm, *}_i \chi L^*_\kappa L_\kappa   \chi_0  A E_j^{\pm}, \quad i, j = 1, 2,
\eeqq
and second 
\beqq\label{eq-pp2}
P_2 = E^{\pm, *}_i \chi L^*_\kappa E  \chi_0 A E_j, \quad i, j = 1, 2.
\eeqq
For $P_1$, we will use results in Section \ref{sec-comp} to find a parametrix which are pseudo-differential operators of order $-1$ on $\mcm_0$. In fact, we will also need to replace $\sigma$ by $\la \sigma, \la\in \mbc$ and consider the holomorphic dependency on $\la$ in order to apply the analytic Fredholm theory, because we do not know the injectivity of the weighted light ray transform in general.  For $P_2$, we show the compactness as in Section \ref{sec-pfthm1}. Finally, we finish the proof using the analytic Fredholm theorem. We split the proof into three subsections.

\subsection{Analysis of $P_1$} \label{sec-part1}    
We first ignore the terms with operator $E$ in \eqref{eq-Xnew} and focus on the parametrix construction. 
 We apply $\chi L_\kappa^*$ to $L_\kappa  \chi_0 Af$ to get  
\beqq\label{eq-u2}
\begin{gathered}
\chi N_\kappa  \chi_0 Af = \chi N_\kappa  \chi_0 AE_1^+ f_1 + \chi N_\kappa  \chi_0 AE_2^+ f_2 
 + \chi N_\kappa  \chi_0 AE_1^- f_1 
 +  \chi N_\kappa  \chi_0 AE_2^-f_2.
 \end{gathered}
\eeqq
Now we apply $E^{+, \ast}_{1}$ and use  Lemma \ref{lm-comp4} to get 
\beqq\label{eq-upara1}
\begin{gathered}
E_1^{+, *}\chi N_\kappa \chi_0  A f = E_1^{+, *} \chi N_\kappa  \chi_0 AE_1^+ f_1 + E_1^{+, *} \chi N_\kappa  \chi_0 AE_2^+ f_2  
+ R_1 f_1 +  R_2 f_2
 \end{gathered}
\eeqq
with $R_1, R_2$ smoothing operators. In the following, we use $R_j, j = 1, 2$ to denote generic smoothing operators which may change line by line. From Lemma \ref{lm-comp4} part (1), we see that $E_1^{+, *} \chi N_\kappa  \chi_0 AE_1^+ \in \Psi^{-1+m}(\mcm_0)$ and $E_1^{+, *} \chi N_\kappa  \chi_0 AE_2^+ \in \Psi^{-2+m}(\mcm_0)$. 

On the other hand, we apply $E_1^{-, \ast}$ to \eqref{eq-u2} to get 
\beqq\label{eq-upara2}
\begin{gathered}
E_1^{-, *}\chi N_\kappa  \chi_0  Af = E_1^{-, *} \chi N_\kappa  \chi_0 AE_1^- f_1 + E_1^{-, *} \chi N_\kappa  \chi_0 AE_2^- f_2  
+ R_1 f_1 +  R_2 f_2
 \end{gathered}
\eeqq
From Lemma \ref{lm-comp4} part (1), we see that $E_1^{-, *} \chi N_\kappa  \chi_0 AE_1^- \in \Psi^{-1+m}(\mcm_0)$ and $E_1^{-, *} \chi N_\kappa  \chi_0 AE_2^- \in \Psi^{-2+m}(\mcm_0)$. Without loss of generality, we assume that the principal symbol of $A$ on $\La_\pm$ is positive. Then it follows from Lemma \ref{lm-wvsym}, the composition results Lemma \ref{lm-comp3} and \ref{lm-comp4}, and the positivity of the symbol of $N_\kappa$ on $\La_1\backslash \La_0$ in Proposition \ref{prop-weiker}   that 
\beq
\begin{gathered}
\sigma(E_1^{+, *} \chi N_\kappa  \chi_0 AE_1^+) > 0, \quad  \sigma( E_1^{-, *} \chi N_\kappa  \chi_0 AE_1^-) > 0\\
\sigma(E_1^{+, *} \chi N_\kappa  \chi_0 AE_2^+ ) >0,  \quad  \sigma( E_1^{-, *} \chi N_\kappa  \chi_0 AE_2^-) <0. 
\end{gathered}
\eeq
Let $Q^+, Q^- \in \Psi^{1-m}(\mcm_0)$ be parametrices for $E_1^{+, *} \chi N_\kappa  \chi_0 AE_1^+$ and $E_1^{-, *} \chi N_\kappa  \chi_0  AE_1^-$ respectively. We know that the principal symbols of $Q^\pm$ are positive. Applying $Q^\pm$ to \eqref{eq-upara1}, \eqref{eq-upara2}, we get 
\beqq\label{eq-upara3}
\begin{gathered}
Q^+ E_1^{+, *}\chi N_\kappa \chi_0  A f =   f_1 + B_+ f_2 + R_1 f_1 +  R_2 f_2 
\end{gathered}
\eeqq
\beqq\label{eq-upara4}
\begin{gathered} 
Q^- E_1^{-, *}\chi N_\kappa \chi_0  A u =  f_1 + B_- f_2 + R_1 f_1 +  R_2 f_2
\end{gathered}
\eeqq
where
\beq
B_+ = Q_+E_1^{+, *} \chi N_\kappa  \chi_0 AE_2^+, \quad B_- = Q_-E_1^{-, *} \chi N_\kappa  \chi_0 AE_2^- 
\eeq
From \eqref{eq-upara3}, \eqref{eq-upara4}, we get 
\beq
Q^+ E_1^{+, *}\chi N_\kappa \chi_0  Af  - Q^- E_1^{-, *}\chi N_\kappa  \chi_0 Af =  (B_+ - B_-) f_2 + R_1  f_1 +  R_2  f_2
\eeq
Note that $B_\pm \in \Psi^{-1}(\mcm_0)$ are elliptic. Also, the principal symbol of $B_+$ is positive but the principal symbol of $B_-$ is negative. Thus $B_+ - B_- \in \Psi^{-1}(\mcm_0)$ is elliptic. Let $W\in \Psi^1(\mcm_0)$ be a parametrix for $B_+ - B_-$. We get 
\beqq\label{eq-fpara1}
WQ^+ E_1^{+, *}\chi N_\kappa \chi_0  Af  - WQ^- E_1^{-, *}\chi N_\kappa  \chi_0 Af =  f_2 + R_1  f_1 +  R_2  f_2
\eeqq
So we solved $f_2$ up to smooth terms. We can use $f_2$ for example in \eqref{eq-upara3} to get  
\beqq\label{eq-fpara2}
\begin{gathered}
Q^+ E_1^{+, *}\chi N_\kappa \chi_0  Af - B_+(WQ^+ E_1^{+, *}\chi N_\kappa  \chi_0 Af  - WQ^- E_1^{-, *}\chi N_\kappa   \chi_0 Af) \\
=   f_1  + R_3 f_1 +  R_4 f_2
\end{gathered}
\eeqq
where $R_3, R_4$ are smoothing operators. 
We are done with the parametrix construction.

\subsection{Compactness of $P_2$}\label{sec-part2}
We repeat the construction with the terms involving operator $E$ in \eqref{eq-Xnew}. Using \eqref{eq-Xnew} and \eqref{eq-fpara1}, we arrive at 
\beqq\label{eq-fpara3}
\begin{gathered}
WQ^+ E_1^{+, *}\chi N_\kappa  \chi_0 Af  - WQ^- E_1^{-, *}\chi N_\kappa  \chi_0 Af =  f_2 + R_1  f_1 +  R_2  f_2 \\
+ WQ^+ E_1^{+, *}\chi L_\kappa^* E \chi_0  A ( E_1 f_1 + E_2 f_2)
 - WQ^- E_1^{-, *}\chi  L_\kappa^* E  \chi_0 A( E_1 f_1 + E_2 f_2)\\
  = f_2 + R_1  f_1 +  R_2  f_2  + R_1' f_1 + R_2'f_2
\end{gathered}
\eeqq
where 
\beqq\label{eq-rp12}
\begin{gathered}
R_1' = WQ^+ E_1^{+, *}\chi L_\kappa^* E  \chi_0 AE_1 - WQ^- E_1^{-, *}\chi  L_\kappa^* E \chi_0  AE_1\\
R_2' = WQ^+ E_1^{+, *}\chi L_\kappa^* E  \chi_0 AE_2 - WQ^- E_1^{-, *}\chi  L_\kappa^* E  \chi_0 AE_2
\end{gathered}
\eeqq
Using \eqref{eq-fpara2} and \eqref{eq-fpara3}, we get  
\beqq\label{eq-fpara4}
\begin{gathered}
Q^+ E_1^{+, *}\chi N_\kappa \chi_0  Af - B_+(WQ^+ E_1^{+, *}\chi N_\kappa  \chi_0 Af  - WQ^- E_1^{-, *}\chi N_\kappa  \chi_0 Af) \\
=   f_1  + R_3 f_1 +  R_4 f_2 + Q^+ E_1^{+, *}\chi  L_\kappa^*  E  \chi_0 A( E_1 f_1 + E_2 f_2) \\
 - B_+R_1' f_1 + B_+R_2' f_2 
  =  f_1  + R_3 f_1 +  R_4 f_2 + R_3' f_1 + R_4'' f_2
\end{gathered}
\eeqq
where 
\beqq\label{eq-rp34}
\begin{gathered}
R_3' = - B_+ R_1' +  Q^+ E_1^{+, *}\chi  L_\kappa^* E  \chi_0 AE_1 \\
R_4' =  B_+R_2'  +  Q^+ E_1^{+, *}\chi L_\kappa^*   E  \chi_0 A E_2
\end{gathered}
\eeqq
We write \eqref{eq-fpara3} and \eqref{eq-fpara4} in the matrix form as 
\beqq\label{eq-matrix}
\begin{gathered}
\begin{pmatrix}
Q^+ E_1^{+, *}\chi N_\kappa  \chi_0 Af - B_+(WQ^+ E_1^{+, *}\chi N_\kappa \chi_0  Af  - WQ^- E_1^{-, *}\chi N_\kappa  \chi_0 Af)\\
WQ^+ E_1^{+, *}\chi N_\kappa \chi_0  Af  - WQ^- E_1^{-, *}\chi N_\kappa  \chi_0 Af
\end{pmatrix}\\
= \begin{pmatrix}
f_1\\
f_2
\end{pmatrix}
+  \begin{pmatrix}
R_3 + R_3' & R_4+ R_4'\\
R_1 + R_1' & R_2 + R_2'
\end{pmatrix} 
\begin{pmatrix}
f_1\\
f_2
\end{pmatrix}  
\end{gathered}
\eeqq
We already know that $R_j, j = 1,2, 3, 4$ are smoothing operators. We show that $R_j', j = 1, 2, 3, 4$ are compact operators among suitable spaces. 

Among the operators $R'_j, j = 1, 2, 3, 4$, there is a common operator $L^*_\kappa E$ see \eqref{eq-rp12}, \eqref{eq-rp34}. We can prove as in Section \ref{sec-pfthm1} that $L^*_\kappa E$ is compact from $H^2(\mcm)$ to $H^3(\mcm)$ (or $\p_t L_\kappa^* E, \p_x L_\kappa^*E$ are compact on $H^2(\mcm)$). Note that the additional weight factor $\kappa$ in $T_1^{-1}$ was considered in Section \ref{sec-pfthm1}. For the other components of $R_j', j = 1, 2, 3, 4$, we know that  $WQ^\pm \in \Psi^{2 - m}(\mcm_0)$ and $Q^+, B_+WQ^\pm \in \Psi^{1 - m}(\mcm_0)$  so they are bounded  operators as
\beqq\label{eq-wqest}
\begin{gathered}
WQ^\pm: H_{\comp}^s(\mcm_0) \rightarrow H_{\loc}^{s - 2 + m}(\mcm_0)\\
Q^+, B_+WQ^\pm:  H_{\comp}^s(\mcm_0) \rightarrow H_{\loc}^{s - 1+ m}(\mcm_0)
\end{gathered}
\eeqq
Also, for $j = 1, 2$ 
\beq
\begin{gathered}
 E_j^*: H_{\comp}^s(\mcm) \rightarrow H_{\loc}^{s  + j - 1}(\mcm_0), \quad  E_j: H_{\comp}^s(\mcm_0) \rightarrow H_{\loc}^{s + j-1}(\mcm)
\end{gathered}
\eeq
Now we look at $R_1'$ on $H^{2+ m}(\mcm_0)$ which can be decomposed as 
\beq
\begin{gathered}
H^{2+m}(\mcm_0) \overset{A E_1}{\longrightarrow}  H^{2}(\mcm) \overset{L_\kappa^* E}{\longrightarrow}  H^3(\mcm) \overset{WQ^\pm E_1^{+, *}\chi }{\longrightarrow}  H^{1+ m}(\mcm_0)
\end{gathered}
\eeq
in which all the operators are bounded and $L_\kappa^*E$ is compact in addition.  So $R_1'$ is compact. We use similar diagram below. 
Next, for $R_2'$ on   $H^{1+ m}(\mcm_0)$, we have
\beq
\begin{gathered}
H^{1+m}(\mcm_0) \overset{A E_2}{\longrightarrow}  H^{2}(\mcm) \overset{L_\kappa^* E}{\longrightarrow}  H^3(\mcm) \overset{WQ^\pm E_1^{+, *}\chi }{\longrightarrow}  H^{1 + m}(\mcm_0)
\end{gathered}
\eeq
 Similarly, for $R_3', R_4'$, we have
\beq
\begin{gathered}
R_3': H^{2+m}(\mcm_0) \overset{A E_1}{\longrightarrow}  H^{2}(\mcm) \overset{L_\kappa^* E}{\longrightarrow}  H^3(\mcm) \overset{Q^\pm E_1^{+, *}\chi }{\longrightarrow}  H^{m+2}(\mcm_0)\\
R_4': H^{1+m}(\mcm_0) \overset{A E_2}{\longrightarrow}  H^{2}(\mcm) \overset{L_\kappa^* E}{\longrightarrow}  H^3(\mcm) \overset{Q^\pm E_1^{+, *}\chi }{\longrightarrow}  H^{m+2}(\mcm_0)
\end{gathered}
\eeq
We reached the conclusion that $R_j', j = 1, 2, 3, 4$ are compact operators.  

\subsection{Completion of the proof}
From \eqref{eq-matrix}, we obtain an operator ${\bf Id}  + {\bf R}$ on $H^{s+1}(\mcx)\times H^s(\mcx)$ where $\mcx$ is a compact subset of $\mcm$ and  
\beqq\label{eq-R}
 {\bf R} =  \begin{pmatrix} 
R_3 + R_3' & R_4+ R_4'\\
R_1 + R_1' & R_2 + R_2'
\end{pmatrix} 
\eeqq
is compact.  To apply the analytic Fredholm theorem, we let $\la\in \mbc$ and replace $\sigma, k$ in \eqref{eq-boltz} by $\la \sigma, \la k.$ We denote the corresponding operator by  ${\bf R}(\la)$.

Observe that  the weight $\kappa$ in the light ray transform $L_\kappa$ \eqref{eq-lrayw} is now holormophic in $\la$. Then the calculation of kernel in Proposition \ref{prop-weiker} shows that the kernel of the normal operator $N_\kappa = L_\kappa^*L_\kappa$ is holomorphic in $\la.$ In particular, the proof of Proposition \ref{prop-normal} shows that $N_\kappa$ is an FIO on $\La_1$ with symbol holomorphic in $\la$. If we follow the construction in Section \ref{sec-comp}, we see that the operators in Lemma \ref{lm-comp3} and Lemma \ref{lm-comp4} are holomorphic in $\la.$  Finally, going through the construction in Subsection \ref{sec-part1}, we see that the operators on the left hand side of \eqref{eq-matrix} are holomorphic in $\la$ and the remainder terms $R_1, R_2, R_3, R_4$ are holomorphic in $\la.$ We remark that it suffices to consider the parametrix construction up to a finite order residue term, that is with $R_j, j = 1, 2, 3, 4$ belonging to $\Psi^{-N}(\mcm_0)$ for $N$ sufficiently large. Then it is clear that $R_j$ are holomorphic in $\la$ and compact on suitable Sobolev spaces.  

Next, in the expression \eqref{eq-rp12} and \eqref{eq-rp34}, we know that $L_\kappa^* E$ is meromorphic in $\la$ as shown in Section \ref{sec-pfthm1}. 
The other operators in \eqref{eq-rp12} and \eqref{eq-rp34} are holomorphic in $\la$, and $R_j', j = 1, 2, 3, 4$ are meromorphic in $\la.$ This proves that ${\bf R}(\la)$ is meromorphic in $\la.$ 

When $\la = 0$, we see that \eqref{eq-matrix} is reduced to 
\beqq\label{eq-matrix1}
\begin{gathered}
\begin{pmatrix}
Q^+ E_1^{+, *}\chi N  \chi_0 Af - B_+(WQ^+ E_1^{+, *}\chi N Af  - WQ^- E_1^{-, *}\chi N  \chi_0 Af)\\
WQ^+ E_1^{+, *}\chi N  \chi_0 Af  - WQ^- E_1^{-, *}\chi N  \chi_0 Af
\end{pmatrix}\\
= \begin{pmatrix}
f_1\\
f_2
\end{pmatrix}
+  \begin{pmatrix}
R_3  & R_4 \\
R_1 & R_2 
\end{pmatrix} 
\begin{pmatrix}
f_1\\
f_2
\end{pmatrix}  
\end{gathered}
\eeqq
where $N = L^*L$ with $L$ the Minkowski light ray transform.  We know that  $L^*: H_{\comp}^{s}(\mcc)\rightarrow H_{\loc}^{s +\ha}(\mbr^{3+1})$ is bounded and 
\beq
E_1^{\pm, *} \chi L^* :   H_{\comp}^{s }(\mcc)\rightarrow H_{\loc}^{s + 1/2}(\mcm)
\eeq
is bounded.  Thus using  estimates \eqref{eq-wqest} and \eqref{eq-matrix1},   we get for $\rho\in \mbr$ that 
\beqq\label{eq-f12}
\begin{gathered}
\|f_1\|_{H^{2+ m}(\mcm_0) } \leq C\| L \chi_0 A(D) f\|_{H^{1 + 3/2}(\mcc) } + C_\rho \|f_1\|_{H^{s+1+\rho}} + C_\rho \|f_2\|_{H^{s+\rho}}\\
\|f_2\|_{H^{1+m}(\mcm_0) } \leq C\| L \chi_0 A(D) f\|_{H^{1 + 3/2}(\mcc) } + C_\rho \|f_1\|_{H^{s+1+\rho}} + C_\rho \|f_2\|_{H^{s+\rho}}
\end{gathered}
\eeqq 
From assumption (3) of Theorem \ref{thm-main3}, we know that   $L \chi_0 A(D)$ is injective. We can use the argument of Theorem 1.1 of \cite{VaWa} to remove the last two terms in each of the equations in \eqref{eq-f12} and obtain 
\beq
\begin{gathered}
\|f_1\|_{H^{2+ m}(\mcm_0) } \leq C\| L \chi_0 A(D) f\|_{H^{1 + 3/2}(\mcc) }, \quad 
\|f_2\|_{H^{1+m}(\mcm_0) } \leq C\| L \chi_0 A(D) f\|_{H^{1 + 3/2}(\mcc) }. 
\end{gathered}
\eeq
This shows the invertiblity of ${\bf Id} + {\bf R}(0)$.   Thus we can apply the analytic Fredholm theorem to conclude that ${\bf Id} + {\bf R}(\la)$ is invertible on $H^{s+1}(\mcx)\times H^s(\mcx)$ for $\la\in \mbc\backslash \mcs$ where $\mcs$ is a discrete set. Therefore, ${\bf Id} + {\bf R}$ with ${\bf R}$ in \eqref{eq-R} is invertible for $\sigma, k$ in an open dense set of $C^\infty \times C^6$.  

To get the stability estimate, we examine the operators in the left hand side of \eqref{eq-matrix}. We know that  $L_\kappa^*: H_{\comp}^{s}(\mcc)\rightarrow H_{\loc}^{s +\ha}(\mbr^{3+1})$ is bounded, see  Proposition \ref{prop-sobo}. We obtain that 
\beq
E_1^{\pm, *} \chi L_\kappa^* :   H_{\comp}^{s }(\mcc)\rightarrow H_{\loc}^{s + 1/2}(\mcm)
\eeq
is bounded. 
Thus using  estimates \eqref{eq-wqest},  the invertibility of ${\bf Id} + {\bf R}$ and the estimate of $A(D),$ we get 
\beq
\begin{gathered}
\|f_1\|_{H^{2+ m}(\mcm_0) } \leq C\| u_T\|_{H^{1 + 3/2}(\mcc) }, \quad 
\|f_2\|_{H^{1+m}(\mcm_0) } \leq C\| u_T\|_{H^{1 + 3/2}(\mcc) }. 
\end{gathered}
\eeq  
This completes the proof of Theorem \ref{thm-main3}. 

\begin{proof}[Proof of Theorem \ref{thm-main2}]
We need to check (3) of Theorem \ref{thm-main3} which follows from the injectivity of the Minkowski light ray transform on $L^1_{\comp}(\mbr^4)$ (hence on $L^2_{\comp}(\mbr^4)$), see \cite[Theorem 8.1]{VaWa}. 
\end{proof}



\end{document}